\documentclass{amsart}

\usepackage{amssymb, amsmath, amsfonts, amsbsy, tikz}
\usepackage{graphicx, array, kbordermatrix,amscd}
\usetikzlibrary{shapes,arrows,matrix,positioning}

\graphicspath{{Figures1/}{Figures2/}}

\DeclareMathOperator{\pa}{pa}

\DeclareMathOperator{\diag}{diag}
\DeclareMathOperator{\trace}{trace}

\DeclareMathOperator{\cov}{Cov}
\DeclareMathOperator{\var}{Var}

\DeclareMathOperator{\Upper}{Upper}
\DeclareMathOperator{\Blockdiag}{Diag}

\definecolor{darkgreen}{rgb}{0,0.4,0}
\definecolor{MyBlue}{rgb}{0,0.08,0.7} 
\definecolor{MyRed}{rgb}{0.85,0.08,0}

\newcommand{\bi}{{\, \color{MyRed}\leftrightarrow}\,}
\newcommand{\toblue}{{\, \color{MyBlue}\to}\,}
\newcommand{\otblue}{{\, \color{MyBlue}\leftarrow}\,}
\renewcommand{\bi}{\leftrightarrow}
\renewcommand{\toblue}{\to}
\renewcommand{\otblue}{\leftarrow}
\newcommand\trans{T}

\newcommand{\lhs}[1]{\text{left}\left({#1}\right)}
\newcommand{\rhs}[1]{\text{right}\left({#1}\right)}

\title[Structural Equation Models]{Algebraic Problems in Structural
  Equation Modeling}   

\author{Mathias Drton}
\address{Department of Statistics\\
            University of Washington
}
\date{September 21, 2016}

\theoremstyle{plain}
\newtheorem{theorem}{Theorem}[section]
\newtheorem{lemma}{Lemma}[section]
\newtheorem{proposition}{Proposition}[section]
\newtheorem{cor}{Corollary}[section]

\theoremstyle{definition}
\newtheorem{definition}{Definition}[section]
\newtheorem{example}{Example}[section]

\newtheorem{ques}{Question}[section]

\theoremstyle{remark}
\newtheorem*{remark}{Remark}

\numberwithin{equation}{section}
\numberwithin{figure}{section}

\newcommand{\ignore}[1]{}

\begin{document}

\begin{abstract}
  The paper gives an overview of recent advances in structural
  equation modeling.  A structural equation model is a multivariate
  statistical model that is determined by a mixed graph, also known as
  a path diagram.  Our focus is on the covariance matrices of linear
  structural equation models.  In the linear case, each covariance is
  a rational function of parameters that are associated to the edges
  and nodes of the graph.  We statistically motivate algebraic
  problems concerning the rational map that parametrizes the
  covariance matrix.  We review combinatorial tools such as the trek
  rule, projection to ancestral sets, and a graph decomposition due to
  Jin Tian.  Building on these tools, we discuss advances in parameter
  identification, i.e., the study of (generic) injectivity of the
  parametrization, and explain recent results on determinantal
  relations among the covariances.  The paper is based on lectures
  given at the 8th Mathematical Society of Japan Seasonal Institute.
\end{abstract}

\maketitle

\setcounter{tocdepth}{1}
\begin{small}
  \vspace{-.4cm}
  \tableofcontents
\end{small}
\clearpage

\part{Structural Equation Models and Questions of Interest}
\addtocontents{toc}{\vspace{.02cm}}{}{}

\section{Motivation}
\label{sec:stat-motiv}

The following example serves well to introduce the statistical models
we will consider.  It features the simplest instance of what
is known as an instrumental variable model.  An empirical study that
shows this type of model `in action' can be found in
\cite{evans:1999}.

\begin{example}
  \label{ex:iv}
  Does a mother's smoking during pregnancy harm the baby?  To answer
  this question researchers conduct a study in which they record, for
  a sample of pregnancies, the baby's birth weight and the average
  number of cigarettes the mom smoked per day during the first
  trimester.  The researchers observe a significant negative
  correlation between the \emph{birth weight} and \emph{smoking} and
  are tempted to conclude that smoking has a negative effect on the
  baby's health, with an increase in the number of cigarettes smoked
  leading to lower birth weight.

  The cigarette companies are not suprised by this finding.  They
  argue, however, that smoking does not harm baby.  Instead, heavier
  smoking is merely a reflection of underlying factors that are the
  true causes of low birth weight.  Such confounding variables could,
  for instance, be of socio-economic nature.  In the context of the
  smoking-lung cancer debate, prominent Statistician Ronald Fisher
  liked to argue that correlations can be attributed to unobserved
  variables of genetic nature \cite{Stolley:1991}.

  Familiar with this type of counter-argument, the researchers
  cleverly recorded a third variable: The \emph{tax rate} on tobacco
  products in the local jurisdictions of the mothers in the sample.
  It is not unreasonable to assume that the tax rate does not have a
  direct effect on the baby's health.  If there is then variation in
  the tax rate and higher taxes have an effect on the amount of
  smoking, then the effect that smoking has on birth weight can
  be estimated in a model that allows for the presence of unobserved
  confounders, as we will see shortly.

  The above narrative suggests a number of cause-effect relations, as
  well as the absence thereof.  Qualitatively these are
  summarized in the graph in Figure~\ref{fig:iv:dag}.  The
  variables in play are the nodes of the graph and cause-effect
  relationships are indicated as directed edges.  The variable $U$
  represents a confounding variable and is unobserved, which we
  emphasize by coloring its edges in red.

  Structural equation models turn the qualitative descriptions of
  causes and effects into quantified functional relationships.  In
  this article, the functional relationships will always be
  \emph{linear}.  The linear structural equation model for the present
  example is based on the following system of \emph{structural
    equations}:
    \begin{align}
      \label{eq:ivX1}
       X_1 &\;=\;  \lambda_{01}
      \phantom{{}\,+\,{\color{MyBlue}\lambda_{12}}X_1\,+
      \,{\color{MyRed}\lambda_{u3}}U}
      \,+\,\varepsilon_1,\\
      \label{eq:ivX2}
      X_2 &\;=\; 
      \lambda_{02}\,+\,{\color{MyBlue}\lambda_{12}}X_1\,+
      \,{\color{MyRed}\lambda_{u2}}U 
      \,+\,\varepsilon_2,\\ 
      \label{eq:ivX3}
      X_3 &\;=\; 
      \lambda_{03}\,+\,{\color{MyBlue}\lambda_{23}}X_2\,+\,
      {\color{MyRed}\lambda_{u3}}U 
      \,+\,\varepsilon_3,\\ 
      \label{eq:ivU}
            U &\;=\;  \lambda_{0u}
      \phantom{{}\,+\,{\color{MyBlue}\lambda_{12}}X_1\,+
      \,{\color{MyRed}\lambda_{u3}}U}
      \,+\,\varepsilon_u.
    \end{align}
    Here, the error terms $\varepsilon_1$, $\varepsilon_2$,
    $\varepsilon_3$, $\varepsilon_u$ are independent random variables
    with zero mean.  The eight coefficients $\lambda_{01}$,
    $\lambda_{02}$, $\lambda_{03}$, $\lambda_{0u}$,
    $\color{MyBlue}\lambda_{12}$, $\color{MyBlue}\lambda_{23}$,
    $\color{MyRed}\lambda_{u2}$, and $\color{MyRed}\lambda_{u3}$ are
    unknown parameters.  Equation~(\ref{eq:ivX1}) indicates that
    variable $X_1$, the tax rate, has expectation $\lambda_{01}$, from
    which it deviates according to the distribution assumed for
    $\varepsilon_1$.  The analogous statement for the unobserved
    confounder $U$ is made in~(\ref{eq:ivU}).  In~(\ref{eq:ivX2}), the
    amount of smoking, denoted $X_2$, is modeled to be a linear
    function of the tax rate and independent noise.
    Similarly,~(\ref{eq:ivX3}) introduces birth weight, denoted $X_3$,
    as a noisy linear function of smoking.

    The quantity of primary interest is the coefficient
    $\color{MyBlue}\lambda_{23}$ that quantifies the relationship
    between smoking and birth weight.  Using data, we can estimate the
    joint distribution and, in particular, the covariance matrix of
    the three observed variables $X_1$, $X_2$ and $X_3$.  Because the
    error terms are independent and have zero means, the
    covariance between $X_3$ and $X_1$ is
    \begin{equation}
      \label{eq:iv}
      \cov[X_1,X_3] \;=\; {\color{MyBlue}\lambda_{23}}\,
      \cov[X_1,X_2].
    \end{equation}
    Hence, as long as $\cov[X_1,X_2]\not=0$, statistical
    inference about $\color{MyBlue}\lambda_{23}$ may be based on the
    ratio of the two covariances in~(\ref{eq:iv}).
\end{example}

\begin{figure}[t]
  \centering
  \begin{center}
    \scalebox{0.75}{
      \begin{tikzpicture}[->,>=triangle 45,shorten >=1pt, auto,thick,
        main
        node/.style={rectangle,fill=gray!20,draw,font=\sffamily\bfseries}]
      
        \node[main node,rounded corners] (1) {$X_1:$ Tax Rate};
        \node[main node,rounded corners] (2) [right=1.5cm of 1]
        {$X_2:$ Mom's Smoking}; \node[main node,rounded corners] (3)
        [right=1.5cm of 2] {$X_3:$ Baby's Weight}; \node[main
        node,rounded corners] (4) [above right=1cm and -.7cm of 2]
        {$U:$ Confounder};
      
        \path[color=black!20!blue,every
        node/.style={font=\sffamily\small}] (1) edge node {} (2) (2)
        edge node {} (3); \path[color=black!20!red,every
        node/.style={font=\sffamily\small}] (4) edge node {} (2) (4)
        edge node {} (3);
      \end{tikzpicture}
    }
  \end{center}
  \caption{Directed graph representing an instrumental variable model.}
  \label{fig:iv:dag}
\end{figure}
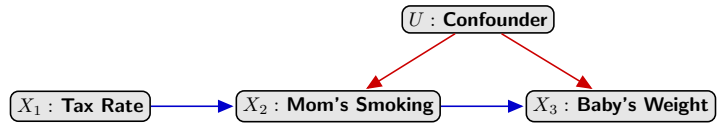

In some applications of structural equation models latent (that is,
unobserved) variables are of direct interest.  For instance, concepts
such as intelligence or depression in psychology are of this nature
and measured only indirectly through other variables such as exam
results or answers in questionnaires.  While problems with explicit
latent variables are ubiquitous \cite{bollen:1989}, we will focus
on models in which the effects of latent variables are summarized and
represented merely in terms of correlations among the error terms in
structural equations.  This representation of dependence induced by
latent variables is discussed in detail in
\cite{MR3338327,koster:2002,pearl:2009,richardson:2002,wermuth:2011}.

\begin{example}
  \label{ex:iv:mixedgraph}
  We take up the instrumental variable model from Example~\ref{ex:iv}.
  The effects of the confounding variable $U$ can be summarized by
  absorbing $U$ into the error terms for equations~(\ref{eq:ivX2})
  and~(\ref{eq:ivX3}).  Define
  \begin{align}
    \label{eq:absorbU}
    {\color{MyRed}\tilde\varepsilon_2} &\;=\; {\color{MyRed}\lambda_{u2}}U 
    {\,+\, \varepsilon_2}, &
    {\color{MyRed}\tilde\varepsilon_3} &\;=\; {\color{MyRed}\lambda_{u3}}U 
    {\,+\, \varepsilon_3}.
  \end{align}
  Retaining only the equations for the observed variables $X_1$,
  $X_2$, and $X_3$, we are left with the equation system:
  \begin{alignat}{2}
    \label{eq:iv:X1mixed}
     X_1 &\;=\; \lambda_{01}
    \phantom{{}\,+\,{\color{MyBlue}\lambda_{12}}X_1}
    \,+\,\varepsilon_1,\\
    \label{eq:iv:X2mixed}
    X_2 &\;=\; 
    \lambda_{02}\,+\,{\color{MyBlue}\lambda_{12}}X_1
    \,+\,{\color{MyRed}\tilde \varepsilon_2}
    ,\\ 
    \label{eq:iv:X3mixed}
    X_3 &\;=\; 
    \lambda_{03}\,+\,{\color{MyBlue}\lambda_{23}}X_2
    \,+\,{\color{MyRed}\tilde \varepsilon_3}
    . 
  \end{alignat}
  However, and this is the significance of the unobserved variable
  $U$, we have now correlated error terms because
  \begin{align}
    \label{eq:iv:w23}
    {\color{MyRed} \omega_{23}} 
    &\;:=\; 
      \cov[{\color{MyRed}
      \tilde\varepsilon_2},{\color{MyRed}\tilde \varepsilon_3}] \;=\;
      \cov[{\color{MyRed}\lambda_{u2}}U +
      \varepsilon_2,{\color{MyRed}\lambda_{u3}}U+ \varepsilon_3]
      \;=\;
      {\color{MyRed}\lambda_{u2}}{\color{MyRed}\lambda_{u3}}\var[U]
      \;\not=\;0.  
  \end{align}
  In the sequel, we will focus on models that are given by equations
  such as~(\ref{eq:iv:X1mixed})-(\ref{eq:iv:X3mixed}), with one
  equation for each observed variable but error terms that may be
  correlated.  Graphically, such models may be represented by a
  \emph{mixed graph} that features directed edges to encode which
  variables appear in each structural equation and bidirected edges
  that indicate possibly nonzero correlations between error terms.
  The mixed graph for the model given
  by~(\ref{eq:iv:X1mixed})-(\ref{eq:iv:X3mixed}) is depicted in
  Figure~\ref{fig:iv:mixed}, which shows the unknown parameters as
  weights for the edges.  At the nodes, we show the variances of the
  error terms, namely, $\omega_{11}=\var[\varepsilon_1]$,
  $\omega_{22}=\var[{\color{MyRed}\tilde\varepsilon_2}]$, and
  $\omega_{33}=\var[{\color{MyRed}\tilde\varepsilon_3}]$.  In the
  statistical literature, the mixed graph for a structural equation
  model is also known as a \emph{path diagram}.

  The ratio $\cov[X_2,X_3]/\var[X_2]$ is the regression
  coefficient when predicting $X_3$ from $X_2$.  We have
  \begin{equation}
    \label{eq:iv:regression-coeff}
    \frac{\cov[X_2,X_3]}{\var[X_2]} \;=\;
    {\color{MyBlue}\lambda_{23}} \;+\; \frac{{\color{MyRed}
        \omega_{23}}}{\var[X_2]}.
  \end{equation}
  Hence, linear regression predicting $X_3$ from $X_2$ only estimates
  the coefficient of interest if ${\color{MyRed} \omega_{23}}=0$, as
  is the case when $X_2$ and $X_3$ do not depend on the latent
  variable $U$.  When ${\color{MyRed} \omega_{23}}\not=0$, the
  relation from~(\ref{eq:iv}), which involves all three variables, is
  needed to recover $\color{MyBlue}\lambda_{23}$.
\end{example}

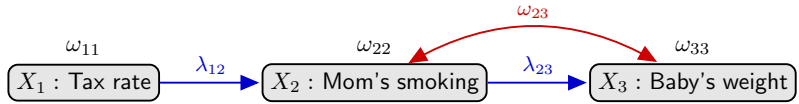
\begin{figure}[t]
  \centering
  \begin{center}
    \scalebox{0.9}{
      \begin{tikzpicture}[->,>=triangle 45,shorten >=1pt, auto,thick,
        main
        node/.style={rectangle,fill=gray!20,draw,font=\sffamily}]
        
        \node[main node,rounded corners] (1) [label=above:{$\omega_{11}$}]{$X_1:$ Tax rate};
        \node[main node,rounded corners] (2) [label=above:{$\omega_{22}$},right=1.5cm of 1]
        {$X_2:$ Mom's smoking}; \node[main node,rounded corners] (3)
        [label=above:{$\omega_{33}$},right=1.5cm of 2] {$X_3:$ Baby's weight};
        
        \path[color=black!20!blue,every
        node/.style={font=\sffamily\small}] (1) edge node
        {$\lambda_{12}$} (2) (2) edge node {$\lambda_{23}$} (3);
        \path[color=black!20!red,<->,every
        node/.style={font=\sffamily\small}, bend left] (2) edge node
        {$\omega_{23}$} (3);
      \end{tikzpicture}
    }
  \end{center}
  \caption{Mixed graph representing an instrumental variable model.}
  \label{fig:iv:mixed}
\end{figure}

The remainder of this paper is organized as follows.
Section~\ref{sec:algebr-comb-nature} introduces linear structural
equation models in full generality.  We then formulate questions of
statistical interest and the algebraic problems they correspond to
(Section~\ref{sec:questions-interest}).  Next, we examine the
interplay between covariance matrices and mixed graphs.  We treat the
so-called trek rule (Section~\ref{sec:trek-rule}) and review useful
results on subgraphs and graph decomposition
(Sections~\ref{sec:induced-subgraphs}
and~\ref{sec:graph-decomposition}).  In Sections~\ref{sec:glob-ident}
and~\ref{sec:gener-ident}, we dive deeper into parameter
identifiability, which here means the question of whether a
coefficient of interest can be recovered from the covariance matrix of
the observed variables.  Finally, we discuss relations among the
entries of the covariance matrix
(Sections~\ref{sec:implicitization}-\ref{sec:verma-constraints}).

\section{Linear Structural Equation
  Models}
\label{sec:algebr-comb-nature}

Let $\varepsilon=(\varepsilon_i:i\in V)$ be a random vector indexed by
a finite set $V$.  Define a new random vector $X=(X_i:i\in V)$ as the
solution to the system of structural equations
\begin{equation}
  \label{eq:sem:general}
  X \;=\; \Lambda^\trans X + \varepsilon,
\end{equation}
where $\Lambda=(\lambda_{ij})\in\mathbb{R}^{V\times V}$ is a matrix of
unknown parameters.  Suppose $\varepsilon$ has covariance matrix
$\Omega=(\omega_{ij})=\var[\varepsilon]$, so $\Omega$ is a positive
definite matrix whose entries are again unknown parameters.  Write $I$
for the identity matrix.  If $I-\Lambda$ is invertible, then the
linear system in~(\ref{eq:sem:general}) is solved uniquely by
$X = (I-\Lambda)^{-\trans}\varepsilon$, which has covariance matrix
\begin{equation}
  \var[X] \;=\; (I - \Lambda)^{-\trans} \Omega (I -
  \Lambda)^{-1} \,\;=:\,\; \phi(\Lambda,
  \Omega).\label{eq:VarX}
\end{equation}
Interesting settings are obtained by restricting the support of
$\Lambda$ and $\Omega$, as is the case in our motivating example.

\begin{example}
  \label{ex:iv:mixedgraph:takeup}
  Consider the setup from Example~\ref{ex:iv:mixedgraph}.  If the
  equation system from~(\ref{eq:iv:X1mixed})-(\ref{eq:iv:X3mixed}) is
  written in vector form as in~(\ref{eq:sem:general}), then the coefficient
  matrix is
  \begin{align}
      \Lambda &=
      \begin{pmatrix}
        0&\lambda_{12}&0\\
        0&0&\lambda_{23}\\
        0&0&0
      \end{pmatrix}.
  \end{align}
  The error covariance matrix is
  \begin{align}
    \Omega &= \var[\varepsilon] =
      \begin{pmatrix}
        \omega_{11} & 0&0\\
        0&\omega_{22} & {\omega_{23}}\\
        0&{\omega_{23}} & \omega_{33}
      \end{pmatrix}.
  \end{align}
  From~(\ref{eq:VarX}), the covariance matrix of 
  $X=(X_1,X_2,X_3)$ is found to be
  \begin{equation}
    \label{eq:iv:Sigma}
    \var[X] \;=\; \begin{pmatrix}
      \omega_{11} & \omega_{11}\lambda_{12} & \omega_{11}\lambda_{12}
      \lambda_{23} \\
      \omega_{11}\lambda_{12} & \omega_{22} +
      \omega_{11}\lambda_{12}^2 & \omega_{23} + 
      \lambda_{23}\sigma_{22}  \\
      \omega_{11}\lambda_{12}\lambda_{23} & \omega_{23} + 
      \lambda_{23}\sigma_{22}& \omega_{33} + 2\omega_{23}\lambda_{23}
      + \lambda_{23}^2\sigma_{22} 
    \end{pmatrix},
  \end{equation}
  where $\sigma_{22}$ denotes the $(2,2)$ entry of $\var[X]$.  The
  relation from~(\ref{eq:iv}) can be confirmed from the $(1,2)$ and
  $(2,3)$ entry of $\var[X]$.
\end{example}

Restrictions on the support of a matrix naturally correspond to a
graph.  Specifically, we adopt mixed graphs because we are dealing
with two matrices, $\Lambda$ and $\Omega$, whose rows and columns are
indexed by the same set $V$.  In structural equation modeling, this
point of view originated in the work of Sewall Wright
\cite{wright:1921,wright:1934}.

A mixed graph with vertex set $V$ is a triple $G=(V,D,B)$ where
$D,B\subseteq V\times V$ are two sets of edges.  The set $D$ comprises
ordered pairs $(i,j)$ that we also denote by $i\toblue j$ to visualize
the fact that such a pair encodes a directed edge pointing from $i$ to
$j$.  Then $i$ is the tail and $j$ is the head of the egde.  The pairs
in $B$ are unordered pairs $\{i,j\}$ that encode bidirected edges that
we also denote by $i\bi j$.  These edges have no orientiation, and
$i\bi j\in B$ if and only if $j\bi i\in B$.  It will be convenient to
call both endpoints $i$ and $j$ heads of $i\bi j$.  In our context,
neither the bidirected part $(V,B)$ nor the directed part $(V,D)$
contain loops, that is, $i\toblue i\not\in D$ and $i\bi i\not\in B$
for all $i\in V$.  If  $(V,D)$ does not contain any
directed cycles $i\toblue\dots\toblue i$, then the mixed graph $G$ is
said to be acyclic.

Let $\mathbb{R}^D$ be the set of real $V\times V$-matrices
$\Lambda=(\lambda_{ij})$ with support in $D$, that is, 
\begin{equation}
  \label{eq:R^D}
      \mathbb{R}^D = \big\{\,\Lambda \in \mathbb{R}^{V
      \times V} : \lambda_{ij} = 0 \ \text{ if } \ i \toblue j \notin D
    \,\big\}.
\end{equation}
Define $\mathbb{R}^D_\mathrm{reg}$ to be the subset of matrices
$\Lambda\in\mathbb{R}^D$ for which $I-\Lambda$ is invertible.  If $G$
is acyclic, then there is a permutation of $V$ that makes $I-\Lambda$
unit upper triangular such that $\det(I-\Lambda)=1$ for all
$\Lambda\in\mathbb{R}^D$ and thus
$\mathbb{R}^D=\mathbb{R}^D_\mathrm{reg}$.  Similarly, let
$\mathit{PD}_V$ be the cone of positive definite symmetric
$V\times V$-matrices $\Omega=(\omega_{ij})$, and define
$\mathit{PD}(B)$ to be the subcone of matrices supported over $B$, that
is, 
\begin{equation}
  \label{eq:PDB}
  \mathit{PD}(B) = \big\{\, \Omega \in \mathit{PD}_V : \omega_{ij} = 0
  \ \text{ if } \
  i \neq j \mbox{ and } i \bi j \notin B\,\big\}.
\end{equation}

Taking the error vector $\varepsilon$ to be Gaussian (or in other
words, to follow a multivariate normal distribution), we arrive at the
following definition of a statistical model for the random vector $X$
that solves~(\ref{eq:sem:general}).  Readers looking for background
such as the fact that linear transformations of a Gaussian random
vector are Gaussian may consult a textbook on multivariate statistics,
e.g., \cite{MR1990662}.

\begin{definition}
  \label{def:sem}
  The \emph{linear structural equation model} given by a mixed graph
  $G=(V,D,B)$ is the family of all multivariate normal distributions
  on $\mathbb{R}^V$ with covariance matrix in the set
  \[
    \mathcal{M}_G \;=\; \left\{ (I-\Lambda)^{-T}\Omega(I-\Lambda)^{-1}
      : \Lambda\in\mathbb{R}^D_\mathrm{reg}, \,
      \Omega\in\mathit{PD}(B)\right\}.
  \]
  The \emph{covariance parametrization} of the model is the map
  \begin{align*}
    \phi_G :  \mathbb{R}^{D} \times \mathit{PD}(B) &\to \mathit{PD}_V, \qquad
    (\Lambda, \Omega) \mapsto (I - \Lambda)^{-\trans} \Omega (I -
                        \Lambda)^{-1},
  \end{align*}
  for which we define the \emph{fiber} of a pair
  $(\Lambda,\Omega)\in\mathbb{R}^D_\mathrm{reg}\times \mathit{PD}(B)$
  to be the preimage
  \begin{equation}
    \label{eq:fiber}
    \mathcal{F}_G(\Lambda,\Omega) \;=\;
    \left\{ 
      (\Lambda',\Omega')\in\mathbb{R}^D_\mathrm{reg}\times \mathit{PD}(B) \::\:
      \phi_G(\Lambda',\Omega')=\phi_G(\Lambda,\Omega)
    \right\}.
  \end{equation}
\end{definition}

As defined, a linear structural equation model does not impose any
restrictions on the mean vector of the normal distributions.
Consequently, the mean vector plays no role in our discussion.  For
instance, in maximum likelihood estimation we may assume without loss
of generality that the mean vector is zero.  Other questions we
consider will directly concern the covariance matrices of the model.
Therefore, we may safely identify a linear structural equation model
with its set of covariance matrices $\mathcal{M}_G$.  On occasion, we
will simply refer to $\mathcal{M}_G$ as the model.

Leaving statistics out of the picture, our interest is in the maps
$\phi_G$, their fibers $\mathcal{F}_G$ and their images
$\mathcal{M}_G$.  Algebra comes into play naturally.

\begin{proposition}
  For any mixed graph $G$, the map $\phi_G$ is a rational map whose
  image $\mathcal{M}_G$ and fibers $\mathcal{F}_G(\Lambda,\Omega)$ are
  semi-algebraic sets.  The map $\phi_G$ is a polynomial map if and
  only if $G$ is acyclic.
\end{proposition}
\begin{proof}
  That $\phi_G$ is rational follows from Cramer's rule for matrix
  inversion.  The domain of $\phi_G$ is a semi-algebraic set and, thus,
  the fibers $\mathcal{F}_G(\Lambda,\Omega)$ are semi-algebraic as
  well.  The Tarski-Seidenberg theorem implies that $\mathcal{M}_G$ is
  semi-algebraic.  If $G=(V,D,B)$ is acyclic, then $\det(I-\Lambda)=1$
  for all $\Lambda\in\mathbb{R}^D$.  Consequently, the entries of
  $(I-\Lambda)^{-1}$ are polynomial in $\Lambda$.  If $G$ is not
  acyclic, then $\det(I-\Lambda)$ is a non-constant polynomial.  The
  Leibniz formula shows that its terms correspond to collections of
  disjoint directed cycles in the graph; compare Theorem 1 in
  \cite{harary:1962}.
\end{proof}


\begin{figure}[t]
  \centering
    \scalebox{0.9}{
      \begin{tikzpicture}[->,>=triangle 45, shorten >=0pt, auto,thick,
        main node/.style={circle,inner
          sep=2pt,fill=gray!20,draw,font=\sffamily}]
        
        \node[main node] (1) {1}; \node[main node] (2) [below
        right=1.25cm and 1.5cm of 1] {2}; \node[main node] (3) [right=
        3.5cm of 1] {3}; \node[main node] (4) [below 
        right=1.25cm and 1.5cm of 3] {4};
        
        \path[color=black!20!blue,every
        node/.style={font=\sffamily\small}] (1) edge node {$\color{MyBlue}\lambda_{12}$} (2) (1)
        edge node {$\color{MyBlue}\lambda_{13}$} (3) (2) edge node {$\color{MyBlue}\lambda_{23}$} (3) (3) edge node {$\color{MyBlue}\lambda_{34}$} (4);
        \path[color=black!20!red,<->,every
        node/.style={font=\sffamily\small}] (2) edge node {$\color{MyRed}\omega_{24}$} (4);
      \end{tikzpicture}
    }
    \caption{An acyclic mixed graph known as the Verma graph.}
  \label{fig:recap}
\end{figure}
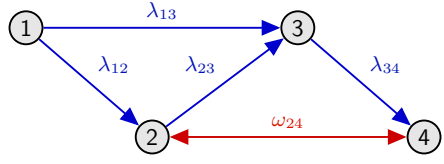

\begin{example}
  \label{ex:recap}
  The mixed graph $G=(V,D,B)$ in Figure~\ref{fig:recap}
  encodes the structural equations
  \begin{center}
    \begin{tabular}{@{}>{$}r<{$}>{$}l<{$}>{$}l<{$}@{}}
      X_1   & =\;\; \lambda_{01}    & \,+\,\varepsilon_1,\\
      X_2   & =\;\; \lambda_{02}\,+\,{\lambda_{12}}X_1  
                                    & \,+\,{\varepsilon_2}\\ 
      X_3   & =\;\; \lambda_{03}\,+\,{\lambda_{13}}X_1 \,+\,{\lambda_{23}}X_2  & \,+\,\varepsilon_3\\
      X_4   & =\;\; \lambda_{04}\,+\,{\lambda_{34}}X_3     & \,+\,{\varepsilon_4}.
    \end{tabular}
  \end{center}
  Only the errors ${\varepsilon_2}$ and
  ${\varepsilon_4}$ may be dependent and
  the error covariance matrix is
  \[
    \Omega \;=\;
    \begin{pmatrix}
      \omega_{11} & 0 &0&0\\
      0&\omega_{22}&0&\omega_{24}\\
      0&0&\omega_{33}&0\\
      0&\omega_{24}&0&\omega_{44}
    \end{pmatrix}.
  \]
  Substracting the coefficient matrix from the identity gives
  \[
    I-\Lambda\;=\;
    \begin{pmatrix}
      1 & -\lambda_{12} & -\lambda_{13} & 0 \\
      0 & 1 & -\lambda_{23} & 0 \\
      0 & 0 & 1 & -\lambda_{34} \\
      0 & 0 & 0 & 1
    \end{pmatrix}
  \]
  with $\det(I-\Lambda)=1$ and inverse
  \[
    (I- \Lambda)^{-1} = 
    \begin{pmatrix}
      1 & \lambda_{12} & \lambda_{13} + \lambda_{12}\lambda_{23} & \lambda_{13} \lambda_{34} + \lambda_{12}\lambda_{23}\lambda_{34} \\
      0 & 1 & \lambda_{23} &  \lambda_{23}\lambda_{34} \\
      0 & 0 & 1 & \lambda_{34} \\
      0 & 0 & 0 & 1
    \end{pmatrix}.
  \]
  To illustrate the form of the map $\phi_G$, we display the coordinate function
  \begin{equation}
     \label{eq:recap-24}
     \phi_G(\Lambda,\Omega)_{24} 
     \;=\;
     \lambda_{12}\lambda_{13} \lambda_{34}\omega_{11} +
     \lambda_{12}^2 \lambda_{23} \lambda_{34}\omega_{11} +
     \lambda_{23} \lambda_{34}\omega_{22} +
     \omega_{24}.
   \end{equation}
\end{example}


\section{Questions of Interest}
\label{sec:questions-interest}

Structural equation models are used to empirically estimate, test and
possibly discover cause-effect relationships among a set of variables.
In estimation and testing, the underlying graph is given.  In
discovery, we seek to estimate the underlying graph, or in other
words, perform model selection.  In this section we give a broad overview
of algebraic problems that arise in the context of these statistical
tasks.  Only some of the problems are treated in the remainder of the
paper, in which we focus on parameter identifiability and polynomial
relations between covariances.

\subsection{Parameter identification}

When considering the model given by the mixed graph $G=(V,D,B)$, a
first question is whether the effects of interest are
\emph{identifiable}, that is, whether they are determined by the
joint distribution of the observed variables.  The importance of the
question is clear: The joint distribution is what can be estimated
from data.  In our setting of linear and Gaussian models, the problem
is equivalent to deciding whether the coefficients $\lambda_{ij}$ in
the linear structural equations can be recovered from the covariance
matrix of the variables.

Different notions of parameter identifiability translate into related
but slightly different algebraic problems.  The most stringent
identifiability property of a model is to have all of its coefficients
$\lambda_{ij}$, $i\to j\in D$, identifiable.  In this case, we seek to
answer the following:

\begin{ques}
  \label{qu:glob-id}
  Is the map $\phi_G$ is injective?
\end{ques}

Injectivity of $\phi_G$ can be decided efficiently, as we will discuss
in Section~\ref{sec:glob-ident}.  However, injectivity can be too
strong of a requirement because all fibers are required to be
singletons with $\mathcal{F}_G(\Lambda,\Omega)=\{(\Lambda,\Omega)\}$.
Indeed, some interesting examples have fibers that are not singletons.

\begin{example}
  \label{ex:iv-generic}
  The map $\phi_G$ fails to be injective when $G$ is the graph for the
  instrumental variable model from Example~\ref{ex:iv:mixedgraph}.
  The relation from~(\ref{eq:iv}) shows that
  $\mathcal{F}_G(\Lambda,\Omega)=\{(\Lambda,\Omega)\}$ if
  $\lambda_{12}\not=0$.  If $\lambda_{12}=0$, however, then the fiber
  is infinite.  Hence, all model parameters are identifiable as long
  as $\lambda_{12}\not=0$.  In the context of
  Example~\ref{ex:iv:mixedgraph}, this requires making an argument
  that higher tax rates impact the amount of smoking.
\end{example}

In the example just given,
$\mathcal{F}_G(\Lambda,\Omega)=\{(\Lambda,\Omega)\}$ for generic
choices of
$(\Lambda,\Omega)\in\mathbb{R}^D_\mathrm{reg}\times \mathit{PD}(B)$.
In this case, we call $\phi_G$ \emph{generically injective}.  We are
led to:
\begin{ques}
  \label{qu:gen-id}
  Is the map $\phi_G$ is generically injective?
\end{ques}

It turns out that generic injectivity is more difficult to decide.
The computational complexity of the problem has not yet been
determined.  In Section~\ref{sec:gener-ident}, we review methods to
decide whether $\phi_G$ is generically injective as well as methods to
decide when the fibers are generically infinite.

When $\phi_G$ is not generically injective, its generic fibers may be
discrete sets.  This property is known as \emph{local identifiability}
in the statistical literature.  We will instead speak of $\phi_G$
being \emph{generically finite-to-one} to highlight that in our case a
discrete fiber is in fact finite because $\phi_G$ is rational.  By
the inverse function theorem, the question of whether $\phi_G$ is
generically finite-to-one is the same as:
\begin{ques}
  \label{qu:rank}
  Does the Jacobian of $\phi_G$ have full column rank?
\end{ques}

The fiber $\mathcal{F}_G(\Lambda,\Omega)$ is defined by the equation
system $\phi_G(\Lambda',\Omega')=\phi_G(\Lambda,\Omega)$.  These
equation systems have a generic number of complex solutions (i.e., the
free entries of $\Lambda$ and $\Omega$ are allowed to be complex
numbers).  

\begin{definition}
  \label{def:id-degree}
  The map $\phi_G$ is algebraically $k$-to-one if the equation systems
  defining its fibers generically have $k$ complex solutions.  We call the
  number $k$ the algebraic degree of identifiability of $G$.
\end{definition}

The degree of identifiability may be determined by Gr\"obner basis
methods (see Section~\ref{sec:gener-ident}).  It is finite if and only
if $\phi_G$ is generically finite-to-one.  The main theorem in
Section~\ref{thm:glob-id} shows that if $\phi_G$ is injective then its
inverse is rational, which is the same as $G$ having degree of
identifiability one.  Currently, there are no combinatorial results
about when the degree is finite but larger than one.  Example 8(b) in
\cite{foygel:draisma:drton:2012} has degree $3$ but fibers whose
cardinality over the reals is either one or three.  To the author's
knowledge, no example has been discovered in which $\phi_G$ is
generically injective over the reals but algebraically $k$-to-one for
$k\ge 2$.

An important question that we will not address in detail is the
identifiability of only a single parameter $\lambda_{ij}$ for a
designated edge of interest $i\to j\in D$.  This amounts to checking
whether in every fiber the coefficient for the edge has only a single
value.  In other words, it must hold that
$\lambda_{ij}'=\lambda_{ij}''$ whenever $(\Lambda',\Omega')$ and
$(\Lambda'',\Omega'')$ are in the same fiber.  In
Example~\ref{ex:iv:mixedgraph}, the fiber of a pair $(\Lambda,\Omega)$
with $\lambda_{12}=0$ is infinite but all pairs $(\Lambda',\Omega')$
such a fiber have $\lambda_{12}'=0$.  Two well-known graphical
methods for identifying a single edge coefficient are the back-door and
the front-door criterion \cite{pearl:2009};  see also
\cite{chen2014,garcia:2010}.

\subsection{Model dimension}
\label{subsec:model-dim}

Statistical tests may be used to assess whether a model is compatible
with empirical data.  At an intuitive level, such tests are based on
computing a distance between data and model and comparing this
distance to typical distances that are obtained when data are
generated from a distribution in the model.  For linear Gaussian
models, a test can be thought off as assessing the distance between
the empirical (or sample) covariance matrix and the model
$\mathcal{M}_G$.  Recall that we have defined $\mathcal{M}_G$ as the
set of covariance matrices.

The challenging part of designing a statistical test is to quantify,
in a probabilistic manner, what typical distances between data and
model are.  Many procedures rely on asymptotic approximations that are
obtained by letting the number of data points grow to infinity.  Under
regularity conditions, limiting distribution theory leads to
consideration of so-called chi-square distributions, which are indexed
by an integer parameter.  In our context, when testing the model given
by the mixed graph $G=(V,D,B)$, the chi-square parameter is set equal
to the codimension of $\mathcal{M}_G$, where we think of
$\mathcal{M}_G$ as embedded in the space of symmetric matrices.  This
gives concrete statistical motivation for:
\begin{ques}
  \label{qu:dim}
  What is the dimension of $\mathcal{M}_G$?
\end{ques}

The model $\mathcal{M}_G$ is parametrized by the coefficients and
covariances associated with the edges in $D$ and $B$ as well as the
variances associated with the nodes in $V$.  Being a subset of the space of
$V\times V$ symmetric matrices, $\mathcal{M}_G$ has
expected dimension
\[
  \min\left\{ |V|+|D|+|B|, \frac{|V|(|V|+1)}{2}\right\}.
\]
The term $|D|+|V|+|B|$ is the count of nodes and edges in $G$.  Since
$\mathcal{M}_G$ is the image of $\phi_G$, its actual dimension is
equal to the rank of the Jacobian of $\phi_G$.  A review of the
connection between dimension and Jacobian in a statistical context is
given in \cite{MR1863967}.  Question~\ref{qu:dim} is tied to parameter
identifiability, most directly to Question~\ref{qu:rank}.  If $\phi_G$
is generically finite-to-one, then $\mathcal{M}_G$ has the expected
dimension $|V|+|D|+|B|$.

\subsection{Covariance equivalence}

Different graphs may induce the same statistical model.  For example,
take $V=\{1,2\}$, and let $G_1$ be the graph with the single edge
$1\to 2$.  Let $G_2$ and $G_3$ be the graphs with single edge
$1\leftarrow 2$ and $1\bi 2$, respectively.  Then
$\mathcal{M}_{G_1}=\mathcal{M}_{G_2}=\mathcal{M}_{G_3}$ as each model
is easily seen to be equal to the entire cone of positive definite
$2\times 2$ matrices.  

From an applied perspective, two different graphs $G$ and $G'$ encode
different scientific/causal hypotheses.  If
$\mathcal{M}_G=\mathcal{M}_{G'}$, then the two hypotheses cannot be
distinguished based on data from a linear and Gaussian structural
equation model.  It is thus useful to be able to decide whether two
graphs $G$ and $G'$ are covariance equivalent, that is, we would like to
be able to answer:

\begin{ques}
  \label{qu:model-equivalence}
  When do two maps $\phi_G$ and $\phi_{G'}$ have the same image?
\end{ques}

Existing results addressing this questions make comparisons between
relations among the entries of the matrices in each model, and we
record:

\begin{ques}
  \label{qu:implicit}
  What are the algebraic relations among the coordinates of $\phi_G$?
\end{ques}

Such relations are also of interest for statistical tests that assess
whether the model given by $G$ is compatible with available data; see,
e.g., \cite{bollen2000tetrad,chen2014,MR2458187}.  We review results
on relations among the covariances in
Sections~\ref{sec:implicitization}-\ref{sec:verma-constraints}.  An
important role is played by determinants that represent probabilistic
conditional independence in a Gaussian random vector.  We note that
models can, in principle, also be distinguished using inequality
constraints.  However, as less is know about inequalities, we do not
treat them here.  Examples of models with latent variables for which a
full semi-algebraic description is available can be found in
\cite{drtonyu:2010,zwiernik;2016}.

\begin{remark}
  As defined above, covariance equivalence is based on data that is
  observational, meaning that it is collected by merely observing the
  considered physical system.  The situation is different when
  experimental data is available, meaning, that data is collected in
  different settings in which the system is subject to various
  experimental interventions.  We will not treat such interventional
  data in this paper.  Interested readers may find discussions of the
  problem in \cite{MR3299409,pearl:2009,spirtes:2000}.  Similarly,
  even for observational data, questions of equivalence differ from
  Question~\ref{qu:model-equivalence} in non-linear models or linear
  models with Gaussian errors \cite{ernest:2016,Shimizu14}.
\end{remark}

\subsection{Maximum likelihood}

The parameters of linear structural equation models are most commonly
estimated using the technique of maximum likelihood.  Suppose we
observe a sample $X^{(1)},\dots,X^{(n)}$ drawn independently from the
multivariate normal distribution with mean vector $\mu$ and covariance
matrix $\Sigma$, which we denote by $\mathcal{N}(\mu,\Sigma)$.  The
joint distribution of the random vectors $X^{(1)},\dots,X^{(n)}$ is
the $n$-fold product of $\mathcal{N}(\mu,\Sigma)$.  The likelihood of
the sample is the value of the joint density of the product
distribution at $(X^{(1)},\dots,X^{(n)})$.  The \emph{likelihood
  function} is the function mapping the pair $(\mu,\Sigma)$ to the
likelihood of the sample.  The \emph{maximum likelihood estimator}
(MLE) of $(\mu,\Sigma)$ under the model given by a mixed graph $G$ is
the maximizer of the likelihood function when restricting $\Sigma$ to
be in $\mathcal{M}_G$.

Define the sample mean vector and the sample covariance matrix as
\begin{equation}
  \label{eq:sample-mean-cov}
\bar X_n \;=\; \frac{1}{n}\sum_{i=1}^n X^{(i)} \quad\text{and}\quad S_n
\;=\; \frac{1}{n} (X^{(i)}-\bar X_n)(X^{(i)}-\bar X_n)^T,
\end{equation}
respectively.  It is convenient to treat the likelihood function on
the log-scale.  With an additive constant omitted and $n/2$ divided
out, the \emph{log-likelihood function} is
\begin{equation*}
  (\mu,\Sigma) \;\mapsto\; - \log\det(\Sigma) -
  \trace\left(\Sigma^{-1} S_n\right) - (\bar
  X_n-\mu)^T\Sigma^{-1}(\bar X_n-\mu).
\end{equation*}
Because the considered models place no constraint on the mean vector,
its MLE is always $\bar X_n$.  The MLE of $\Sigma$ solves the
problem of maximizing the function
\begin{equation}\label{eq:likelihood}
  \ell(\Sigma) = - \log\det(\Sigma) -
  \trace\left(\Sigma^{-1} S_n\right)
\end{equation}
subject to $\Sigma\in\mathcal{M}_G$.  Using the covariance
parametrization, the MLE is found by maximizing $\ell\circ \phi_G$.  A
key problem is then understanding the existence and uniqueness of
the MLE.   We record:

\begin{ques}
  \label{qu:mle-exist}
  For which sample covariance matrices $S_n$ does the likelihood
  function $\ell\circ \phi_G$ achieve its maximum?
\end{ques}

Graphical models theory solves Question~\ref{qu:mle-exist} when
$G=(V,D,B)$ is an acyclic digraph, i.e., has $B=\emptyset$ and $D$
without directed cycles \cite{lauritzen:1996}.  More generally, it is
well known that $\ell\circ\phi_G$ is bounded when $S_n$ is positive
definite but this is not necessary \cite{fox:2014}.  An issue that is
not well explored is the fact that even if $\ell\circ\phi_G$ is
bounded it may fail to achieve its maximum as the model
$\mathcal{M}_G$ need not be closed.  For instance, the model in
Example~\ref{ex:two-ivs} is not closed.  We remark that
Question~\ref{qu:mle-exist} is closely related to a positive definite
matrix completion problem that arises in ML estimation for other types
of graphical models \cite{buhl:1993,Sullivant:Gross,MR3014306}.

In some models, the MLE is known to admit a closed-form expression as
a rational function of the data.  Such models have \emph{maximum
  likelihood (ML) degree} equal to one, in the sense of the following:
\begin{ques}
  The MLE of $\Sigma$ in model $\mathcal{M}_G$ is an algebraic
  function of the data.  What is the degree of this function?
\end{ques}

An introduction to the notion of ML degree is given in \cite[Chapter
2]{oberwolfach}.  Here, we merely not that the ML degree is one when
$G$ is an acyclic digraph.  More general models may have higher ML
degree and a log-likelihood function with more than one local maximum.
We exemplify this for a model discussed in detail in
\cite{drton:richardson:2004}.

\begin{figure}[t]
  \centering
  \scalebox{.9}{
    \begin{tikzpicture}[->,>=triangle 45,shorten >=1pt, auto,thick,
      main node/.style={circle,inner
        sep=2pt,fill=gray!20,draw,font=\sffamily}]
          
      \node[main node] (1) {1}; \node[main node] (2) [right = 1.75cm
      of 1] {2}; \node[main node] (3) [right = 1.75cm of 2] {3};
      \node[main node] (4) [right = 1.75cm of 3] {4};
      \path[color=black!20!blue,every
      node/.style={font=\sffamily\small}] (1) edge node[above]
      {$\lambda_{12}$} (2) (4) edge node {$\lambda_{43}$} (3);
      \path[color=black!20!red,<->,every
      node/.style={font=\sffamily\small}] (2) edge (3);
    \end{tikzpicture}
  }
  \caption{Mixed graph for a bivariate seemingly unrelated regressions
    model.}
  \label{fig:sur}
\end{figure}
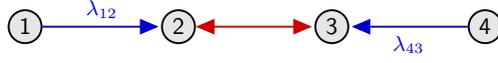
 
\begin{example}
  \label{ex:sur}
  Suppose we have data with sample covariance matrix
  \[
    S_n= \kbordermatrix{ & X_1 & X_2 & X_3 & X_4\cr X_1 & 8 & -5 & 10 &
      3 \cr X_2 &-5 & 27 & 4 & 49 \cr X_3 &10 & 4 & 21 & 24 \cr X_4 &
      3 & 49 & 24 & 114}.
  \]
  The matrix is positive definite such that the log-likelihood
  function $\ell$ from~(\ref{eq:likelihood}) is bounded above on the
  entire cone of positive definite matrices.  More precisely, $\ell$ has
  compact level sets, that is, for any constant $c\in\mathbb{R}$
  the set of positive definite matrices $\Sigma$ with $\ell(\Sigma)\ge
  c$ is compact \cite{MR1990662}.

  Let $G$ be the graph depicted in Figure~\ref{fig:sur}.  It is not
  difficult to show that the parametrization $\phi_G$ admits a
  rational inverse.  Let $\Sigma=(\sigma_{ij})$ satisfy
  $\Sigma=\phi_G(\Lambda,\Omega)$ with
  $\Lambda=(\lambda_{ij})\in\mathbb{R}^D$ and
  $\Omega=(\omega_{ij})\in\mathit{PD}(B)$.  Then
  \begin{align}
    \label{eq:sur-global-id}
    \lambda_{12}&=\frac{\sigma_{12}}{\sigma_{11}}, 
    &
      \lambda_{43} &=\frac{\sigma_{34}}{\sigma_{44}},
  \end{align}
  and the entries of $\Omega=(I-\Lambda)^T\Sigma(I-\Lambda)$ are
  rational functions of $\Sigma$ as well.  All the rational functions
  are defined on the entire cone of positive definite matrices because
  $\sigma_{11},\sigma_{44}>0$.  It is also clear that the considered
  map $\phi_G$ is proper, that is, compact subsets of the positive
  definite cone have compact preimages under $\phi_G$.  It follows
  that $\ell\circ\phi_G$ has compact level sets and, thus,
  achieves its maximum on the open set
  $\mathbb{R}^D\times \mathit{PD}(B)$.

  The critical points of $\ell\circ\phi_G$ satisfy a rational equation
  system, in which the determinant of $\phi_G(\Lambda,\Omega)$ appears
  in the denominator.  Since the directed part of $G$ is acyclic the
  determinant is equal to the determinant of $\Omega$.  Clearing the
  denominator yields a polynomial equation system.  Saturating the
  system with respect to the determinant removes infeasible solutions
  with $\det(\Omega)=0$.  Computing a lexicographic Gr\"obner basis
  after the saturation shows that the critical points
  $(\Lambda,\Omega)$ solve the equation
  \begin{multline*}
    10583160 \,{ \lambda_{12}^5}+43115307\,
    { \lambda_{12}^4}+72738452\,
    {\lambda_{12}^3}\\
    +55482894\, { \lambda_{12}^2}+8437660\,
    { \lambda_{12}}-4703765 \;=\; 0.
  \end{multline*}
  All other entries of $\Lambda$ and also $\Omega$ solve linear
  equations whose coefficients depend on $\lambda_{12}$ and the data.
  We conclude that the MLE of $\Sigma$ is an algebraic function of
  degree 5.  We note that the displayed equation for $\lambda_{12}$
  has three real roots and, thus, is not solvable by
  radicals.
\end{example}





\subsection{Model singularities}

As noted in Section~\ref{subsec:model-dim}, the distributions of test
statistics are frequently approximated using asymptotic theory.  For
so-called likelihood ratio tests, this asymptotic theory can be
thought of as assessing infinitesimal distances between a positive
semidefinite data matrix and the given model $\mathcal{M}_G$.  The
data matrix is generated from a distribution that corresponds to a
particular point in $\mathcal{M}_G$.  At a smooth point of
$\mathcal{M}_G$, the distribution of the infinitesimal distance is
given by a chi-square distribution which is the distribution of the
distance between a Gaussian random vector and a linear space.  At
singular points, the distribution is determined by the tangent cone at
the considered point \cite{drton:lrt}.  Singularities also impact
other approaches such as Wald tests \cite{MR3449776}, and it is
important to clarify:

\begin{ques}
  Is the image of $\phi_G$ a smooth manifold?  If not, what are the
  tangent cones of the image?
\end{ques}

We do not address the question explicitly in this paper.  However,
whenever $\phi_G$ is injective (see Section~\ref{sec:glob-ident}) its
image is smooth.  Indeed, when $\phi_G$ is injective it has a rational
inverse whose domain of definition includes the cone of all positive
definite matrices \cite{drton:2011}.  The next example illustrates
that not all models are smooth.

\begin{figure}[t]
        \centering
        \scalebox{.9}{
          \begin{tikzpicture}[->,>=triangle 45, shorten >=0pt,
            auto,thick,
            main node/.style={circle,inner sep=2pt,
              fill=gray!20,draw,font=\sffamily}] 
            
            \node[main node,rounded corners] (3) 
            {3};
            \node[main node,rounded corners] (1) 
            [above left =0.5cm and 1.5cm of 3] {1}; 
            \node[main node,rounded corners] (2) 
            [below left =0.5cm and 1.5cm of 3] {2}; 
            \node[main node,rounded corners] (4)
            [right=1.5cm of 3] {4};
            
            \path[color=black!20!blue,every node/.style={font=\sffamily\small}] 
            (1) edge node[above] {} (3)
            (1) edge node[above] {} (2)
            (2) edge node[below] {} (3)
            (3) edge node {} (4);
            \path[color=black!20!red,<->,every
            node/.style={font=\sffamily\small},bend left] 
            (3)  edge node {} (4);
          \end{tikzpicture}
        }
  \caption{Mixed graph of a model with two instruments.}
  \label{fig:two-ivs}
\end{figure}
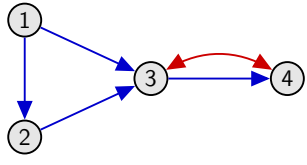

\begin{example}
  \label{ex:two-ivs}
  Consider the mixed graph $G=(V,D,B)$ from Figure~\ref{fig:two-ivs}.
  Let $\Sigma\in\mathbb{R}^{4\times 4}$ be a positive definite matrix.
  Define the matrix
  \begin{align*}
   \Sigma_{\{3,4\}.\{1,2\}} 
    &= \Sigma_{\{3,4\},\{1,2\}}\left(
      \Sigma_{\{1,2\},\{1,2\}}\right)^{-1},
  \end{align*}
  and the Schur complement
  \begin{align*}
   \Sigma_{\{3,4\},\{3,4\}.\{1,2\}} 
    &=  \Sigma_{\{3,4\},\{3,4\}}-\Sigma_{\{3,4\},\{1,2\}}\left(
        \Sigma_{\{1,2\},\{1,2\}}\right)^{-1}\Sigma_{\{1,2\},\{3,4\}}.
  \end{align*}
  Change coordinates to the triple of $2\times 2$ matrices
  \[
    \left(\Sigma_{\{1,2\},\{1,2\}},\; \Sigma_{\{3,4\}.\{1,2\}} ,\;\Sigma_{\{3,4\},\{3,4\}.\{1,2\}} \right).
  \]
  If $\Sigma=\phi_G(\Lambda,\Omega)$ for
  $\Lambda=(\lambda_{ij})\in\mathbb{R}^D$ and
  $\Omega=(\omega_{ij})\in\mathit{PD}(B)$, then
  \begin{align*}
    \Sigma_{\{1,2\},\{1,2\} }
    &=
      \begin{pmatrix}
        \omega_{11} &
        \lambda_{12}\omega_{11}\\
        \lambda_{12}\omega_{11} &\omega_{22}
      \end{pmatrix},\\
    \Sigma_{\{3,4\}.\{1,2\}} 
    &=
      \begin{pmatrix}
        \lambda_{13} & \lambda_{23}\\
        \lambda_{13}\lambda_{34} &\lambda_{23}\lambda_{34}
      \end{pmatrix},\\
    \Sigma_{\{3,4\},\{3,4\}.\{1,2\}} 
    &=
      \begin{pmatrix}
        \omega_{33} & \omega_{34}+\lambda_{34}\omega_{33}\\
        \omega_{34}+\lambda_{34}\omega_{33} &
        \omega_{44}+2\lambda_{34}\omega_{34}+\lambda_{34}^4\omega_{33} 
      \end{pmatrix}.
  \end{align*}
  We observe that $\Sigma$ is in the (topological) closure of
  $\mathcal{M}_G$ if and only if $\Sigma_{\{3,4\}.\{1,2\}}$ is a
  matrix of rank at most one.  In order for $\Sigma$ to be in
  $\mathcal{M}_G$ it must also hold that the second row of
  $\Sigma_{\{3,4\}.\{1,2\}}$ is zero only if the first row is zero.

  Geometrically, the closure of $\mathcal{M}_G$ is 
  equivalent to the product of two cones of positive definite $2\times
  2$ matrices and the set of $2\times 2$
  matrices of rank at most one.  The latter set is singular at the
  zero matrix.  For more details see the related example in
  \cite[Exercise 6.4]{oberwolfach}.
\end{example}

\subsection{Singularities of fibers}

Finally, without going into any detail, we note that it is also of
statistical interest to study the geometry of the fibers
$\mathcal{F}_G(\Lambda,\Omega)$.  In particular, the resolution of
singularities of $\mathcal{F}_G(\Lambda,\Omega)$ is connected to asymptotic
approximations in Bayesian approaches to model selection.  

Bayesian methods assess the goodness-of-fit of a model by integrating
the likelihood function with respect to a prior distribution.  In
models with many parameters, the integration is over a
high-dimensional domain and, thus, constitutes a difficult numerical
problem.  While carefully tuned Monte Carlo methods can be effective,
it can also be useful to invoke asymptotics.  For large sample size
$n$, the integrated likelihood function behaves like a Laplace
integral.  Under regularity conditions, a Laplace approximation can
yield accurate approximations that have been used in many applications
\cite{MR924875}.  However, the models considered here may also lead to
singular Laplace integrals for which asymptotic approximations are
more involved.

Asymptotic expansions for singular Laplace integrals are well-studied
\cite{MR966191}.  The work of Sumio Watanabe brings the ideas to bear
in the statistical context \cite{MR2554932}.  For several practically
relevant settings, it has become tractable to determine or bound the
\emph{real log-canonical threshold} and its multiplicity, which
determine how the integrated likelihood scales with the sample size
$n$.  This information can be used in model selection \cite{sBIC}.
Computing real log-canonical thresholds for data generated under the
distribution with covariance matrix $\phi_G(\Lambda,\Omega)$ requires
careful study of the singularities of the fiber
$\mathcal{F}_G(\Lambda,\Omega)$.  Bounds on the thresholds can be
obtained from cruder information such as the dimension of the fiber.


\addtocontents{toc}{\vspace{-.25cm}}{}{}
\part{Treks, Subgraphs and Decomposition}
\addtocontents{toc}{\vspace{.02cm}}{}{}

\section{Trek Rule}
\label{sec:trek-rule}

In solving the problems from Section~\ref{sec:questions-interest}, it
is desirable to exploit the connection between the covariance
parametrization $\phi_G$ of a structural equation model and the
underlying mixed graph $G=(V,D,B)$.  The trek rule that we present in
this section makes the connection precise and is behind results that
allow one to answer some of the questions we posed with efficient
algorithms.

It is natural to expect the covariance between random variables $X_i$
and $X_j$ to be determined by the semi-walks between the nodes $i$ and
$j$ in the graph $G$.  A semi-walk is an alternating sequence of nodes
from $V$ and edges from either $D$ or $B$ such that the endpoints of
each edge are the nodes immediately preceding and succeeding the edge
in the sequence.  In other words, a semi-walk is a walk that
uses bidirected or directed edges but is allowed to traverse directed
edges in the `wrong direction'.   As we will see, only
special semi-walks contribute to the covariance.

\begin{definition}
  A \emph{trek} $\tau$ from \emph{source} $i$ to \emph{target} $j$ is a
  semi-walk from $i$ to $j$ whose consecutive edges do not have any
  colliding arrowheads.  In other words, $\tau$ is a sequence of the
  form
  \begin{enumerate}
  \item[(a)] $i\leftarrow i_{l}\leftarrow \dots
      \leftarrow i_1\leftarrow
      i_0\longleftrightarrow j_0\to
      j_1\to \dots \to
      j_{r}\to j$, or
      \smallskip
  \item[(b)] $i\leftarrow
                       i_{l}\leftarrow \dots \leftarrow
                       i_1\xleftarrow[]{\hspace{0.74cm}}
                       i_0\xrightarrow[]{\hspace{0.74cm}}  j_1\to
                       \dots \to j_{r}\to j$. 
  \end{enumerate}
  A trek has a \emph{left-} and a \emph{right-hand side}, denoted
  $\lhs{\tau}$ and $\rhs{\tau}$, respectively.  We have
  $\lhs{\tau}=\{i_0,\dots,i_{l},i\}$ and
  $\rhs{\tau}=\{j_0,\dots,j_{r},j\}$ in case (a), and
  $\lhs{\tau}=\{i_0,\dots,i_{l},i\}$ and
  $\rhs{\tau}=\{i_0,j_1,\dots,j_{r},j\}$ in case (b).  In case (b) the
  \emph{top} node $i_0$ belongs to both sides.  A trek $\tau$ from $i$
  to $i$ may have no edges, in which case $i$ is the top node,
  $\lhs{\tau}=\rhs{\tau}=\{i\}$.  We call such a trek trivial.
\end{definition}

In an acyclic graph, if we think of directed edges pointing
`downward', then a trek takes us up and/or down a `mountain'.  
Any directed path is a trek, in which case $|\lhs{\tau}|=1$ or
$|\rhs{\tau}|=1$ depending on the direction in which the path is
traversed.  A trek may contain the same node on both its left- and
right-hand sides.  If the graph contains a cycle, then the left- or
right-hand side of $\tau$ may contain this cycle, possibly repeated.

For a trek $\tau$ that contains no bidirected edge and has top node
$i_0$, define a trek monomial as
\[
\tau(\Lambda,\Omega) = \omega_{i_0i_0}\prod_{k\to l\in \tau} \lambda_{kl}.
\]
For a trek $\tau$ that contains a bidirected edge $i_0\bi j_0$, define
the trek monomial as
\[
\tau(\Lambda,\Omega) = \omega_{i_0j_0}\prod_{k\to l\in \tau} \lambda_{kl}.
\]
The following rule expresses the covariance matrix $\Sigma$ as a
summation over treks \cite{spirtes:2000,wright:1921,wright:1934}.  We
write $\mathcal{T}(i,j)$ for the set of all treks from $i$ to $j$.

\begin{theorem}[Trek rule]
  Let $G=(V,D,B)$ be any mixed graph, and let $\Lambda\in\mathbb{R}^D$
  and $\Omega\in\mathit{PD}(B)$.  Then the covariances are
  \begin{equation}
    \label{eq:trek-rule}
    \phi_G(\Lambda,\Omega)_{ij} = \sum_{\tau\in\mathcal{T}(i,j)}
    \tau(\Lambda,\Omega), \qquad i,j\in V.
  \end{equation}
\end{theorem}

Some clarification is in order.  If $G$ is acyclic, then the summation
in~(\ref{eq:trek-rule}) is finite and yields a polynomial.  If $G$
contains a directed cycle, then the right-hand side
of~(\ref{eq:trek-rule}) may yield a power series as shown in
Example~\ref{ex:trek-rule-cyclic} below.  Under assumptions on the
spectrum of $\Lambda$, the power series converges and yields the value
of $\phi_G(\Lambda,\Omega)_{ij}$.  These spectral conditions are also
needed to give cyclic models an interpretation of representing
observation of an equilibrium \cite{fish:corr,Lacerda08}.
This said, it is also useful to treat the series as a formal power
series.  If so desired, a combinatorial description can also be given
for a rational expression for $\phi_G(\Lambda,\Omega)_{ij}$; compare
\cite{MR3044565}.

\begin{proof}[Proof of the trek rule]
  Writing $(I-\Lambda)^{-1}=I+\Lambda+\Lambda^2+\dots$, we observe
  that
  \begin{align}
    \label{eq:directed-paths}
    ((I- \Lambda)^{-1})_{ij} &= \sum_{\tau \in \mathcal{P}(i,j)}
                               \prod_{k\rightarrow l \in \tau} \lambda_{kl},
  \end{align}
  where $\mathcal{P}(i,j)$ is the set of directed paths from $i$ to
  $j$ in $G$.  If $G$ is acyclic, then $\Lambda^m=0$ for all
  $m\ge |V|$, and the geometric series of matrices has only finitely
  many nonzero terms.  If $G$ is cyclic the geometric series is
  infinite and converges if and only if all eigenvalues of $\Lambda$
  have magnitude less than 1.  Now, observe that a product of three
  entries of $(I- \Lambda)^{-T}$, $\Omega$, and $(I- \Lambda)^{-1}$,
  respectively, corresponds to the concatenation of two directed paths
  at a common top node or by joining them with a bidirected edge.  A
  top node represents a diagonal entry of $\Omega$, and a bidirected
  edge an off-diagonal entry of $\Omega$.
\end{proof}

\begin{figure}[t]
  \centering
    \begin{center}
    \scalebox{0.8}{
      \begin{tikzpicture}[->,>=triangle 45, shorten >=0pt, auto,thick,
        main node/.style={circle,inner
          sep=2pt,fill=gray!20,draw,font=\sffamily}]

        \node[main node] (a) {2}; 
        \node[main node,label={[black!20!red]above right:{$\omega_{11}$}}]
          (b) 
          [right= 2cm of a] {1}; 
        \node[main node] (c) [right= 2cm of b] {3}; 
        \node[main node] (d) [right= 2cm of c] {4}; 
        
        \path[color=black!50!green,every
        node/.style={font=\sffamily\small}] (b) edge node[above]
        {$\lambda_{12}$} (a);
        \path[color=black!20!blue,every
        node/.style={font=\sffamily\small}] (b) edge node
        {$\lambda_{13}$} (c) (c) edge node
        {$\lambda_{34}$} (d);
        
        \node[main node] (a2) [below=.55cm of a] {2}; 
        \node[main node,label={[black!20!red]above right:{$\omega_{11}$}}] (b2)
          [right= 2cm of a2] {1}; 
        \node[main node] (c2) [right= 2cm of b2] {2}; 
        \node[main node] (d2) [right= 2cm of c2] {3}; 
        \node[main node] (e2) [right= 2cm of d2] {4}; 
        
        \path[color=black!50!green,every
        node/.style={font=\sffamily\small}] (b2) edge node[above]
        {$\lambda_{12}$} (a2);
        \path[color=black!20!blue,every
        node/.style={font=\sffamily\small}] (b2) edge node
        {$\lambda_{12}$} (c2) (c2) edge node
        {$\lambda_{23}$} (d2) (d2) edge node
        {$\lambda_{34}$} (e2);
        
        \node[main node,label={[black!20!red]above right:{$\omega_{22}$}}] (a3) [below=.55cm of a2] {2}; 
        \node[main node] (b3)
          [right= 2cm of a3] {3}; 
        \node[main node] (c3) [right= 2cm of b3] {4}; 
        
        \path[color=black!20!blue,every
        node/.style={font=\sffamily\small}] (a3) edge node[above]
        {$\color{MyBlue}\lambda_{12}$} (b3) (b3) edge node
        {$\color{MyBlue}\lambda_{13}$} (c3);

         \node[main node] (a4) [right=2cm of c3] {2}; 
        \node[main node] (b4)
          [right= 2cm of a4] {4}; 
        
        \path[color=black!20!red,<->,every
        node/.style={font=\sffamily\small}] (a4) edge node
        {$\color{MyRed}\omega_{24}$} (b4);
      \end{tikzpicture}
    }
  \end{center}
  \caption{Four treks in the graph from Figure~\ref{fig:recap}.}
  \label{fig:recap:treks}
\end{figure}
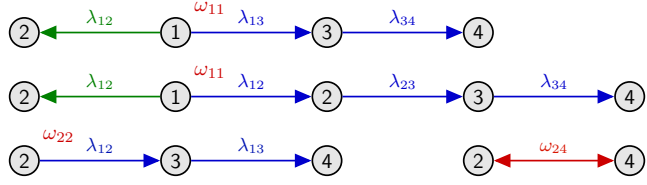

\begin{example}
  In Example~\ref{ex:recap}, the coordinate function $(\phi_G)_{24}$
  is a polynomial with four terms; see~(\ref{eq:recap-24}).  
  The 
  terms
  correspond to the four treks shown in Figure~\ref{fig:recap:treks}.
\end{example}

\begin{example}
  \label{ex:trek-rule-cyclic}
  Let $G$ be the graph from Figure~\ref{fig:two-ivs-cyclic}, which
  contains the directed cycle $2\to 3\to 4\to 2$.  Due to this cycle,
  $\det(I-\Lambda)=1-\lambda_{23} \lambda_{34} \lambda_{42}$.  
  As an example of a coordinate of $\phi_G$ we select
  \begin{align*}
    \phi_G(\Lambda,\Omega)_{24} &= \frac{1}{(1-\lambda_{23} \lambda_{34}
                                  \lambda_{42})^2}
\left[\lambda_{12}^2 \lambda_{23} \lambda_{34} \omega_{11}+\lambda_{12} \lambda_{13} \lambda_{34} \omega_{11} (\lambda_{23} \lambda_{34} \lambda_{42}+1)\right.\\
&\quad \left. +\lambda_{13}^2 \lambda_{34}^2 \lambda_{42} \omega_{11}+\lambda_{23} \lambda_{34} \omega_{22}+\lambda_{34}^2 \lambda_{42} \omega_{33}+2 \lambda_{34} \lambda_{42} \omega_{34}+\lambda_{42} \omega_{44}\right].
  \end{align*}
  To understand how this rational formula relates to the trek rule,
  let us focus on the treks from 2 to 4 that use the bidirected edge
  $2\bi 4$.  There are then two treks for which both left and right
  side are self-avoiding paths, namely,
  \begin{align*}
    \tau_1&:2 \leftarrow 4 \bi 3 \to 4, & 
    \tau_2&:2 \leftarrow 4 \leftarrow 3 \bi 4.
  \end{align*}
  Both of these treks yield the same monomial and together contribute
  the term $2\lambda_{34}\lambda_{42}\omega_{34}$ to $(\phi_G)_{24}$.
  All other treks from 2 to 4 that use edge $2\bi 4$ are obtained by
  inserting directed cycles into $\tau_1$ or $\tau_2$.  For instance,
  inserting one cycle on the left- and one on the right-hand side of
  $\tau_1$ gives
  \[
    2 \leftarrow 4 \leftarrow 3\leftarrow 2\leftarrow 4 \bi 3 \to 4
    \to 2\to 3\to 4.
  \]
  The monomials associated with these treks are 
  \[
    \lambda_{34}\lambda_{42}\omega_{34} \left(\lambda_{23}
      \lambda_{34} \lambda_{42} \right)^k, \qquad k=1,2,\dots. 
  \]
  The monomial for exponent $k$ arises from $k+1$ different treks;
  $l=0,1,\dots,k$ cycles are inserted on the left, the remaining $k-l$
  cycles are inserted on the right.  Hence, the contribution to
  $(\phi_G)_{24}$ made by all treks from 2 to 4 that use edge $2\bi 4$
  is
  \[
    2\sum_{k=0}^\infty (k+1)   \lambda_{34}\lambda_{42}\omega_{34} \left(\lambda_{23}
      \lambda_{34} \lambda_{42} \right)^k \;=\;
    \frac{2\lambda_{34}\lambda_{42}\omega_{34}}{(1- \lambda_{23} 
      \lambda_{34} \lambda_{42})^2},
  \]
  assuming that $|\lambda_{23}\lambda_{34}\lambda_{42}|<1$.
  This explains one of the terms in the rational expression for
  $(\phi_G)_{24}$.  The reasoning for the other terms is analogous.
\end{example}

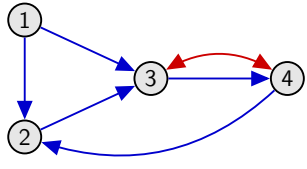
\begin{figure}[t]
        \centering
        \scalebox{.9}{
          \begin{tikzpicture}[->,>=triangle 45, shorten >=0pt,
            auto,thick,
            main node/.style={circle,inner sep=2pt,
              fill=gray!20,draw,font=\sffamily}] 
            
            \node[main node,rounded corners] (3) 
            {3};
            \node[main node,rounded corners] (1) 
            [above left =0.5cm and 1.5cm of 3] {1}; 
            \node[main node,rounded corners] (2) 
            [below left =0.5cm and 1.5cm of 3] {2}; 
            \node[main node,rounded corners] (4)
            [right=1.5cm of 3] {4};
            
            \path[color=black!20!blue,every node/.style={font=\sffamily\small}] 
            (1) edge node[above] {} (3)
            (1) edge node[above] {} (2)
            (2) edge node[below] {} (3)
            (3) edge node {} (4);
            \path[color=black!20!blue,bend left,every node/.style={font=\sffamily\small}]
            (4) edge node {} (2);
            \path[color=black!20!red,<->,every
            node/.style={font=\sffamily\small},bend left] 
            (3)  edge node {} (4);
          \end{tikzpicture}
        }
  \caption{A cyclic mixed graph.}
  \label{fig:two-ivs-cyclic}
\end{figure}


\section{Induced Subgraphs and Principal Submatrices}
\label{sec:induced-subgraphs}

Suppose $X=(X_i:i\in V)$ follows the linear structural equation model
given by mixed graph $G=(V,D,B)$, so $\var[X]=\phi_G(\Lambda,\Omega)$ for some
$\Lambda\in\mathbb{R}^D_\mathrm{reg}$ and $\Omega\in\mathit{PD}(B)$.
Let $A\subseteq V$ be a subset of nodes.  Then the covariance matrix
of the subvector $X_A=(X_i:i\in A)$ is obtained by projecting to the
relevant principal submatrix, that is,
\begin{equation}
  \label{eq:principal-submatrix}
  \var[X_A]\;=\;
  \phi_G(\Lambda,\Omega)_{A,A}.  
\end{equation}
The resulting map $(\Lambda,\Omega)\mapsto \phi_G(\Lambda,\Omega)_{A,A}$ may be
complicated, even when $\phi_G$ is not.

\begin{example}
  \label{ex:factor-analysis}
  Suppose $G=(V,D,\emptyset)$ is a directed graph with $i\to j\in D$
  if and only if $i\notin A$ and $j\in A$; the graph is thus
  bipartite.  Then the image of $\phi_G(\Lambda,\Omega)_{A,A}$ is the
  set of covariance matrices of a factor analysis model with
  $|V\setminus A|$ factors.  Factor analysis models have complicated
  geometric structure, particularly when $V\setminus A$ has more than
  two elements \cite{MR2772115,MR2299716,MR2652316,MR2504939}.
  Open problems remain even for $|V\setminus A|=1$ when
  allowing additional directed edges among the nodes in $A$
  \cite{MR3466188}.
\end{example}

For a general mixed graph $G=(V,D,B)$, let $D_A=D\cap (A\times A)$ be
the set of directed edges with both endpoints in $A$.  Similarly, let
$B_A\subset B$ be the set of bidirected edges that have both endpoints
in $A$.  The subgraph induced by $A$ is the mixed graph
$G_A=(A,D_A,B_A)$.  Example~\ref{ex:factor-analysis} and also already
Example~\ref{ex:iv} show that the covariance matrices obtained from
$\phi_{G_A}$ generally differ from those obtained by projecting
$\phi_G$ onto the $A\times A$ submatrix.  However, as we now
emphasize, induced subgraphs are relevant in a special case.

Define the set of \emph{parents} of a node $i\in V$ as
\[
  \pa(i) \;=\; \{j\in V\::\: j  \to i \}.
\]
A set of nodes $A\subseteq V$ is \emph{ancestral} if $i\in A$ implies
that $\pa(i)\subseteq A$.  The terminology indicates that such a set
contains all its ancestors, where an ancestor is a node with a
directed path to some node in $A$.  Ancestral sets are obtained by
recursively removing sink nodes.  We define node $i$ to be a
\emph{sink node} of $G$ if it is a sink of the directed part $(V,D)$,
that is, if all directed edges incident to $i$ have their head at $i$.

\begin{theorem}
  \label{thm:ancestral-subgraphs}
  Let $G=(V,D,B)$ be a mixed graph, and let $G_A$ be the subgraph
  induced by an ancestral set $A\subset V$.  Then for all
  $\Lambda\in\mathbb{R}^D_\mathrm{reg}$ and $\Omega\in\mathit{PD}(B)$,
  we have
  \[
    \phi_{G_A}(\Lambda_{A, A},\Omega_{A, A}) \;=\;
    \left[\phi_{G}(\Lambda,\Omega)\right]_{A, A}.
  \]
\end{theorem}
\begin{proof}
  Let $i,j\in A$.  By the trek-rule, the $(i,j)$ entry of
  $\phi_G(\Lambda,\Omega)$ is given by summing the monomials associated
  to treks from $i$ to $j$ in $G$.  Because $A$ is ancestral, a trek from $i$
  to $j$ in $G$ cannot leave $A$.  Hence, the set of treks from $i$ to
  $j$ in $G$ coincides with the set of treks from $i$ to $j$ in the
  subgraph $G_A$.  Applying the trek-rule to $G_A$ yields the claim.
\end{proof}

\begin{example}
  \label{ex:remove-sink}
  Take up Example~\ref{ex:iv:mixedgraph:takeup}.  Clearly, node 3 is a
  sink in the graph from Figure~\ref{fig:iv:mixed}.  Hence, the set
  $\{1,2\}$ is ancestral.  Inspecting the matrix displayed
  in~(\ref{eq:iv:Sigma}), we see that removing the third row and
  column yields the matrix for the induced subgraph $1\to 2$.
\end{example}

\section{Graph Decomposition}
\label{sec:graph-decomposition}

We now present a very useful graph decomposition, which allows one to
address several questions of interest by considering smaller
subgraphs.  The decomposition for acyclic mixed graphs was introduced
in work of Jin Tian \cite{tian:2005,tian:pearl:2002}.  Here we
formulate the natural extension to possibly cyclic mixed graphs.

Consider partitioning the vertex set of a mixed graph $G=(V,D,B)$ in
two ways.  The first partition is given by the connected components of
the bidirected part $(V,B)$.  The second partition is obtained from
the strongly connected components of the digraph $(V,D)$.  For two
distinct nodes to belong to the same strongly connected component,
there must be directed paths in either direction between them.  Let
$\mathcal{C}(G)$ be the finest common coarsening of the two
partitions.  Two nodes $i\not=j$ are in the same block of
$\mathcal{C}(G)$ if and only if they are connected by a path that uses
only edges that are bidirected or part of some directed cycle.  Note
that $G$ is acyclic if and only if all strongly connected components
are singleton sets, in which case the blocks of $\mathcal{C}(G)$ are
simply the connected components of the bidirected part $(V,B)$.

For a block $C\in\mathcal{C}(G)$, define 
\[
  V[C] := C \cup \bigcup_{i\in C} \pa(i)
\]
to be the union of the block and all parents of nodes in the block.
Let $D[C] = D \cap (V[C] \times C)$ be the set of directed edges with
head in $C$.   Let $B[C]$ be the set of bidirected with both
endpoints in $C$.

\begin{definition}
  The graphs $G[C] = (V[C], D[C], B[C])$, $C\in\mathcal{C}(G)$, form a
  decomposition of $G$, and we refer to them as the
  \emph{mixed components} of $G$.
\end{definition}

A graph decomposition partitions the edge set.  As we are working with
mixed graphs both edge sets are partitioned in the decomposition.  We
note that the set $V[C]\setminus C$ is the set of sources nodes of
$G[C]$.  Here, we define a source node as a node that is a tail on all
edges it is incident to, with the convention that both endpoints of a
bidirected edge are heads.

\begin{example}
  \label{ex:tian-decomp-graph}
  The graph in Figure~\ref{fig:tian-decomp} has the bidirected
  components $\{1,4\}$, $\{3\}$, and $\{2,5\}$.  The strong components
  of the directed part are $\{1\}$, $\{2,3\}$, $\{4\}$, and $\{5\}$.
  The finest common coarsening of the two partitions is
  $\{ \{1,4\}, \{2,3,5\}\}$.  The graph thus has two mixed components
  with vertex sets $V[\{2,3,5\}]=\{1,2,3,4,5\}$ and
  $V[\{1,4\}]=\{1,2,3,4\}$.  The mixed component with vertex set
  $V[\{2,3,5\}]$ contains all edges with a head in $\{2,3,5\}$, and
  the second mixed component contains all edges with a head in
  $\{1,4\}$.  The components are depicted in Figure~\ref{fig:tian-decomp}.
\end{example}

For $C\in\mathcal{C}(G)$, define the projection
$\pi^\to_C:\mathbb{R}^{V\times V}\to\mathbb{R}^{V[C]\times V[C]}$ by
setting
\[
  \pi^\to_C(\Lambda)_{ij} =
  \begin{cases}
    \lambda_{ij} &\text{if} \ j\in C,\\
    0 &\text{if} \ j\in V[C]\setminus C.
  \end{cases}
\]
Define a second map
$\pi^\bi_C:\mathbb{R}^{V\times V}\to\mathbb{R}^{V[C]\times V[C]}$ by
\[
  \pi^\bi_C(\Omega)_{ij} =
  \begin{cases}
    \omega_{ij} &\text{if} \ i,j\in C,\\
    1 &\text{if} \ i=j\in V[C]\setminus C,\\
    0 &\text{otherwise}.
  \end{cases}
\]
It projects onto the $C\times C$ submatrix of $\Omega$ and then adds
an identity matrix as a diagonal block over $V[C]\setminus C$.  Define
a subset of $\mathit{PD}(B[C])\subset
\mathbb{R}^{V[C]\times V[C]}$ as
\[
\mathit{PD}_I(B[C])\;=\; \left\{\Omega=(\omega_{ij})\in\mathit{PD}(B[C])
  \::\: \omega_{ii}=1 \;\text{ if }\;  i\in V[C]\setminus C \right\}.
\]
Then we have
\[
\pi^\to_C:\mathbb{R}^D\to\mathbb{R}^{D[C]} \quad\text{and}\quad
\pi^\to_C:\mathit{PD}(B)\to\mathit{PD}_I(B[C])
\]
because $G[C]$ is a subgraph of $G$.  With
$\pi_C=(\pi^\to_C,\pi^\bi_C)$, we obtain the isomorphism
\[
\pi=(\pi_C)_{C\in\mathcal{C}(G)} \::\:
 \mathbb{R}^D\times
\mathit{PD}(B) \:\to\: \prod_{C\in\mathcal{C}(G)} \mathbb{R}^{D[C]}\times\mathit{PD}_I(B[C]).
\]

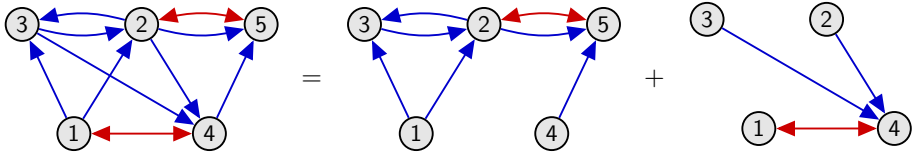
\begin{figure}[t]
  \centering
  \begin{tabular}{m{3.7cm}m{0.15cm}m{3.7cm}m{0.15cm}m{3cm}}
  \scalebox{0.9}{
    \begin{tikzpicture}[->,>=triangle 45, shorten >=0pt,
      auto,thick, main node/.style={circle,inner
        sep=2pt,fill=gray!20,draw,font=\sffamily}]
      
      \node[main node] (3) {3}; 
      \node[main node] (2) [right=1.25cm of 3] {2}; 
      \node[main node] (5) [right=1.25cm of 2] {5}; 
      \node[main node] (1) [below right=1.25cm and 0.4cm of 3] {1};
      \node[main node] (4) [below left=1.25cm and 0.4cm of 5] {4};
      
      \path[color=black!20!blue] 
      (1) edge node {} (2) 
      (3) edge node {} (4)
      (1) edge node {} (3)
      (2) edge node {} (4)
      (4) edge node {} (5);

      \path[color=black!20!blue, bend right=15]
      (2) edge node {} (3)
      (2) edge node {} (5)
      (3) edge node {} (2);

      \path[color=black!20!red,<->, bend left=15] 
      (2) edge node {} (5);
      \path[color=black!20!red,<->] 
      (1) edge node {} (4);
    \end{tikzpicture}
  }
    &$=$ &
  \scalebox{0.9}{
    \begin{tikzpicture}[->,>=triangle 45, shorten >=0pt,
      auto,thick, main node/.style={circle,inner
        sep=2pt,fill=gray!20,draw,font=\sffamily}]
      
      \node[main node] (3) {3}; 
      \node[main node] (2) [right=1.25cm of 3] {2}; 
      \node[main node] (5) [right=1.25cm of 2] {5}; 
      \node[main node] (1) [below right=1.25cm and 0.4cm of 3] {1};
      \node[main node] (4) [below left=1.25cm and 0.4cm of 5] {4};
      
      \path[color=black!20!blue] 
      (1) edge node {} (2) 
      (1) edge node {} (3)
      (4) edge node {} (5);

      \path[color=black!20!blue, bend right=15]
      (2) edge node {} (3)
      (2) edge node {} (5)
      (3) edge node {} (2);

      \path[color=black!20!red,<->, bend left=15] 
      (2) edge node {} (5);
    \end{tikzpicture}
  }
    & $+$ &
  \scalebox{0.9}{
    \begin{tikzpicture}[->,>=triangle 45, shorten >=0pt,
      auto,thick, main node/.style={circle,inner
        sep=2pt,fill=gray!20,draw,font=\sffamily}]
      
      \node[main node] (3) {3}; 
      \node[main node] (2) [right=1.25cm of 3] {2}; 
      \node[main node,draw=none,fill=white] (5) [right=1.25cm of 2]
      {$\phantom{5}$}; 
      \node[main node] (1) [below right=1.25cm and 0.4cm of 3] {1};
      \node[main node] (4) [below left=1.25cm and 0.4cm of 5] {4};
      
      \path[color=black!20!blue] 
      (3) edge node {} (4)
      (2) edge node {} (4);

      \path[color=black!20!red,<->] 
      (1) edge node {} (4);
    \end{tikzpicture}
    }
  \end{tabular}
  \caption{A mixed graph is decomposed into its two mixed
    components.}
  \label{fig:tian-decomp}
\end{figure}

\begin{example}
  \label{ex:tian-decomp-projections}
  Let $G$ be the graph from Figure~\ref{fig:tian-decomp}, which has
  $\mathcal{C}(G)=\{\{1,4\},\{2,3,5\}\}$.  A matrix in $\mathbb{R}^D$
  is of the form
  \begin{align*}
    \Lambda=
    \begin{bmatrix}
      0&\lambda_{12}&\lambda_{13}&0&0\\
      0&0&\lambda_{23}&\lambda_{24}&\lambda_{25}\\
      0&\lambda_{32}&0&\lambda_{34}&0\\
      0&0&0&0&\lambda_{45}\\
      0&0&0&0&0
    \end{bmatrix}.
  \end{align*}
  Its projections are
  \begin{align*}
    \pi^\to_{\{2,3,5\}}(\Lambda)
    &=    
      \begin{bmatrix}
        0&\lambda_{12}&\lambda_{13}&0&0\\
        0&0&\lambda_{23}&0&\lambda_{25}\\
        0&\lambda_{32}&0&0&0\\
      0&0&0&0&\lambda_{45}\\
      0&0&0&0&0
    \end{bmatrix},
    &
    \pi^\to_{\{1,4\}}(\Lambda)
      &=
    \begin{bmatrix}
      0&0&0&0&0\\
      0&0&0&\lambda_{24}&0\\
      0&0&0&\lambda_{34}&0\\
      0&0&0&0&0\\
      0&0&0&0&0
    \end{bmatrix}.
    \end{align*}
    An error covariance matrix in $\mathit{PD}(B)$ has the form
    \begin{align*}
      \Omega=\begin{bmatrix}
        \omega_{11}&0&0&\omega_{14}&0\\
        0&\omega_{22}&0&0&\omega_{25}\\
        0&0&\omega_{33}&0&0\\
        \omega_{14}&0&0&\omega_{44}&0\\
        0&\omega_{25}&0&0&\omega_{55}
      \end{bmatrix},
    \end{align*}
    and we have
    \begin{align*}
      \pi^\bi_{\{2,3,5\}}(\Omega)
      &=
        \begin{bmatrix}
          1&0&0&0&0\\
          0&\omega_{22}&0&0&\omega_{25}\\
          0&0&\omega_{33}&0&0\\
        0&0&0&1&0\\
        0&\omega_{25}&0&0&\omega_{55}
      \end{bmatrix},
                 &
                   \pi^\bi_{\{1,4\}}(\Omega)
                   &=
              \begin{bmatrix}
        \omega_{11}&0&0&\omega_{14}&0\\
        0&1&0&0&0\\
        0&0&1&0&0\\
        \omega_{14}&0&0&\omega_{44}&0\\
        0&0&0&0&1
      \end{bmatrix}.
    \end{align*}
\end{example}

With these preparation in place we may state Tian's theorem as
follows.  Recall that $\mathit{PD}_V$ is our symbol for the cone of
positive definite $V\times V$ matrices.

\begin{theorem}
  \label{thm:tian-decomposition}
  Let $G=(V,D,B)$ be a mixed graph with mixed components
  $G[C]=(V[C],D[C],B[C])$ for $C\in\mathcal{C}(G)$.  Then there is an
  invertible map $\tau$ such that the following diagram commutes:
  \[
    \begin{CD}
      \mathbb{R}^D_\mathrm{reg}\times\mathit{PD}(B)  @>\mspace{50mu}\phi_G\mspace{50mu}>> \mathit{PD}_V\\
      @AAA \mspace{-605mu} @VV{\pi}V \mspace{-75mu} @AAA \mspace{-1mu} @VV{{\tau}}V\\
      \prod_{C\in\mathcal{C}(G)}
      \mathbb{R}^{D[C]}_\mathrm{reg}\times\mathit{PD}_I(B[C])
      @>(\phi_{G[C]})_{C\in\mathcal{C}(G)}>> \prod_{C\in\mathcal{C}(G)}
      \mathit{PD}_{V[C]}
    \end{CD}
    \medskip
  \]
  In other words,
  $\tau\circ \phi_G=(\phi_{G[C]}\circ\pi_C)_{C\in\mathcal{C}(G)}$.  Both $\tau$
  and its inverse are rational maps, defined on all of $\mathit{PD}_V$
  and all of $\prod_{k=1}^m \mathit{PD}_{V_k}$, respectively.
\end{theorem}

Below we give a linear algebraic proof that
makes $\tau$ and its rational nature explicit.  Alternatively, a proof in
probabilistic notation could be given by generalizing the proof of 
\cite[Lemma 1]{tian:pearl:2002}.  For the generalization, the nodes in
the setup of \cite{tian:pearl:2002} may be replaced by the strong
components of the directed part $(V,D)$.

Theorem~\ref{thm:tian-decomposition} is a very useful result as
questions about $\phi_G$ can be answered by studying, one by one, the
maps $\phi_{G[C]}$ for the mixed components.  The fact that $\tau$ and
$\tau^{-1}$ are rational is important.  For instance, it allows one to
obtain precise algebraic information about parameter identifiability
in the sense of Definition~\ref{def:id-degree}.
    
\begin{cor}
  \label{cor:id-decomposition}
  The degree of identifiability of a mixed graph $G$ is the product of
  the degrees of identifiability of its mixed components $G[C]$,
  $C\in\mathcal{C}(G)$.  In particular, $\phi_G$ is (generically)
  injective if and only if each $\phi_{G[C]}$ is so, for
  $C\in\mathcal{C}(G)$.
\end{cor}

Our proof of Theorem~\ref{thm:tian-decomposition} is presented in
terms of Cholesky decompositions.  When applied to a Gaussian
covariance matrix, the Cholesky decomposition corresponds to factoring
the multivariate normal density into a product of conditional
densities, which is the connection to the probabilistic setting of
\cite{tian:pearl:2002}.  We begin by stating a lemma on 
uniqueness and sparsity in block-Cholesky decomposition.

If $\mathcal{C}$ is a partition of a finite set $V$, then we write
$\Blockdiag(\mathcal{C})$ for the space of block-diagonal matrices.
So, $A=(a_{ij})\in\mathbb{R}^{V\times V}$ is in
$\Blockdiag(\mathcal{C})$ if and only if $a_{ij}=0$ whenever $i$ and
$j$ are in distinct blocks of $\mathcal{C}$.  If we order the blocks
of the partition as $\mathcal{C}=\{C_1,\dots,C_k\}$, then we may
define a space of strictly block upper-triangular matrices
$\Upper(C)$, which contains 
$A=(a_{ij})\in\mathbb{R}^{V\times V}$ if
and only if $a_{ij}=0$ whenever $i\in C_u$ and $j\in C_v$ with
$u\ge v$.

\begin{lemma}
  \label{lem:chol-block-unique}
  Let $\Sigma\in\mathit{PD}_V$, and let $\mathcal{C}$ be a partition
  of $V$, with ordered blocks.
  \begin{enumerate}
  \item[(i)] There exist unique matrices $A\in\Upper(\mathcal{C})$ and
    $\Delta\in\Blockdiag(\mathcal{C})$ such that
    \[
      \Sigma = (I-A)^{-T} \Delta (I-A)^{-1}.
    \]
    The matrix $\Delta$ has positive definite diagonal blocks.
  \item[(ii)] Let $\mathcal{C}'$ be a second partition of $V$ that is
    coarser than $\mathcal{C}$. 
    If $\Sigma\in \Blockdiag(\mathcal{C}')$, then the matrix $A$ from
    (i) satisfies $A\in\Blockdiag(\mathcal{C}')$.
  \end{enumerate}
\end{lemma}
\begin{proof}
  (i) A block-LDL decomposition yields $\Sigma=(I+L)\Delta (I+L)^T$ for
  unique $L^T\in \Upper(\mathcal{C})$ and
  $\Delta\in\Blockdiag(\mathcal{C})$.  Unit block upper-triangular
  matrices form a group and, thus, $(I+L)^{-T}=I-A$ for
  $A\in \Upper(\mathcal{C})$.
  (ii) 
  The claim is a consequence of the way fill-in occurs when computing
  the Cholesky decomposition of a sparse matrix
  \cite[Section 4.1]{vandenberghe}.
\end{proof}

\begin{proof}[Proof of Theorem~\ref{thm:tian-decomposition}]
  We first show that, for every block $C\in\mathcal{C}(G)$, there
  exists a map $\tau_C$ such that
  $\tau_C\circ\phi_G=\phi_{G[C]}\circ\pi_C$.  Let
  $\Lambda\in\mathbb{R}^D_{\rm reg}$, $\Omega\in\mathit{PD}(B)$, and
  $\Sigma = \phi_G(\Lambda,\Omega)$.  Then the claim is that
  \begin{align}
    \tau_C(\Sigma)
    \label{eq:proof-to-show}
    &= 
    \begin{bmatrix}
      I-
      \begin{pmatrix}
        0 & \Lambda_{\overline{C},C}\\
        0 & \Lambda_{C,C} 
      \end{pmatrix}
    \end{bmatrix}^{-T}
    \begin{pmatrix}
      I & 0\\
      0 & \Omega_{C,C}
    \end{pmatrix}
    \begin{bmatrix}
      I-
      \begin{pmatrix}
        0 & \Lambda_{\overline{C},C} \\
        0 & \Lambda_{C,C}
      \end{pmatrix}
    \end{bmatrix}^{-1}.
  \end{align}

  Let $\mathcal{C}(D)$ be the partition of $V$ given by the strongly
  connected components of $(V,D)$.  Order the blocks of
  $\mathcal{C}(D)$ topologically as $W_1,\dots,W_k$ such that the
  existence of a directed path from a node in $W_u$ to a node in $W_v$
  implies that $v\ge u$.  By
  Lemma~\ref{lem:chol-block-unique}(i), there are
  $A\in\Upper(\mathcal{C}(D))$ and
  $\Delta\in\Blockdiag(\mathcal{C}(D))$ such that
  \begin{equation}
    \label{eq:proof-chol}
  \Sigma = (I-A)^{-T} \Delta (I-A)^{-1}
  \end{equation}
  Letting $\overline{C}=V[C]\setminus C$, we define
  \begin{equation}
    \label{eq:proof:claimed-W}
    \tau_C(\Sigma) = 
    \begin{bmatrix}
      I-
      \begin{pmatrix}
        0 & A_{\overline{C},C} \\
        0 & A_{C,C}
      \end{pmatrix}
    \end{bmatrix}^{-T}
    \begin{pmatrix}
      I & 0\\
      0 & \Delta_{C,C}
    \end{pmatrix}
    \begin{bmatrix}
      I-
      \begin{pmatrix}
        0 & A_{\overline{C},C}\\
        0 & A_{C,C} 
      \end{pmatrix}
    \end{bmatrix}^{-1}.
  \end{equation}
  
  If $G$ is acyclic then $\Lambda$ is strictly upper-triangular under
  a topological ordering and, thus,
  $\Lambda\in\Upper(\mathcal{C}(D))$.  When $G$ has directed cycles,
  then $\Lambda$ is block upper-triangular but not strictly so.
  Hence, we consider the block-diagonal matrix
  \begin{equation*}
    \label{eq:proof:make-upper-tri}
    \Delta_\Lambda\;=\;\diag( I-\Lambda_{W,W} : W\in\mathcal{C}(D)),
  \end{equation*}
  which is invertible because $\det(I-\Lambda)=\det(\Delta_\Lambda)$ and
  $\Lambda\in\mathbb{R}^D_\mathrm{reg}$.  Hence,
  \begin{equation}
    \label{eq:proof-make-unitriang}
    \Sigma=\big[ (I-\Lambda) \Delta_\Lambda^{-1}\big]^{-T} \big[
    \Delta_\Lambda^{-1}\Omega\Delta_\Lambda^{-1}\big]\big[ (I-\Lambda)
    \Delta_\Lambda^{-1}\big]^{-1}.
  \end{equation}
  Because $\Delta_\Lambda,\Omega\in\Blockdiag(\mathcal{C}(G))$, we
  have
  \begin{equation*}
    \label{eq:proof:Omega-tilde}
    \tilde\Omega \;=\; \Delta_\Lambda^{-T}\Omega\Delta_\Lambda^{-1} \;\in\;
    \Blockdiag(\mathcal{C}(G)) .
  \end{equation*}
  Moreover, 
  due to the block upper-triangular shape of $\Lambda$, 
  \begin{equation*}
    \label{eq:proof:Lambda-tilde}
    \tilde\Lambda \;=\; I-(I-\Lambda) \Delta_\Lambda^{-1}  \;\in\;
    \Upper(\mathcal{C}(D)) .
  \end{equation*}
  By Lemma~\ref{lem:chol-block-unique}(i) and (ii), there are
  $\Delta_\Omega\in\Blockdiag(\mathcal{C}(D))$ and
  $U\in\Upper(\mathcal{C}(D))\cap\Blockdiag(\mathcal{C}(G))$ such that
  \begin{equation}
    \label{eq:proof:Omega-tilde-chol}
    \tilde\Omega = (I-U)^{-T}\Delta_\Omega(I-U)^{-1}.
  \end{equation}
  Combining~(\ref{eq:proof-make-unitriang})
  and~(\ref{eq:proof:Omega-tilde-chol}) gives
  \begin{equation}
    \label{eq:proof:param-chol}
    \Sigma=
    \big[(I-\tilde\Lambda)(I-U)\big]^{-T}\Delta_\Omega
    \big[(I-\tilde\Lambda)(I-U)\big]^{-1}, 
  \end{equation}
  where
    $(I-\tilde\Lambda)(I-U)=I-(\tilde\Lambda+U-\tilde\Lambda U)$
  with $\tilde\Lambda+U-\tilde\Lambda U\in \Upper(\mathcal{C}(D))$.  

  By the uniqueness in Lemma~\ref{lem:chol-block-unique}(i), equations
  (\ref{eq:proof-chol}) and (\ref{eq:proof:param-chol}) imply that
  \begin{align}
    \label{eq:proof-Delta-eqn}
    \Delta&=\Delta_\Omega, \\
    \label{eq:proof-A-eqn}
    A&= \tilde\Lambda+U-\tilde\Lambda U.
  \end{align}
  Since $U\in\Blockdiag(\mathcal{C}(G))$, we have that
  \begin{equation}
    \label{eq:proof:zeros-in-C}
    U_{V\times C} = 
    \begin{pmatrix}
         U_{(V\setminus C)\times C}\\ U_{C\times C}
    \end{pmatrix}
    =    
    \begin{pmatrix}
      0\\ U_{C\times C}
    \end{pmatrix}.
  \end{equation}
  Therefore, by~(\ref{eq:proof-A-eqn}),
  $A_{C,C} = \tilde\Lambda_{C,C}+U_{C,C}-\tilde\Lambda_{C,C} U_{C,C}$,
  and we deduce that
  \begin{equation}
    \label{eq:proof:A-princ-submatrix-factor}
    I- A_{C,C} = (I-\tilde\Lambda_{C,C})(I-U_{C,C}) \quad\text{and}\quad
    A_{\overline{C},C} = \tilde\Lambda_{\overline{C},C}(I-U_{C,C}).
  \end{equation}
  Moreover, 
  \begin{equation}
    \label{eq:proof:tildeOmega-princ}
    \tilde\Omega_{C,C} = (I-U_{C,C})^{-T}\Delta_{C,C}(I-U_{C,C})^{-1},
  \end{equation}
  which follows from~(\ref{eq:proof-Delta-eqn}) and the fact that
  \begin{equation*}
    \label{eq:proof:I-C-partition}
    \big[(I-U)^{-1}\big]_{V,C} = 
    \begin{pmatrix}
      0\\ \big[(I-U)^{-1}\big]_{C,C}
    \end{pmatrix}= 
    \begin{pmatrix}
      0 \\ (I-U_{C,C})^{-1}
    \end{pmatrix},
  \end{equation*}
  which in turn follows from $U$ being
  in $\Blockdiag(\mathcal{C}(G))$.  
  
  Substituting the formulas from
  (\ref{eq:proof:A-princ-submatrix-factor})
  into~(\ref{eq:proof:claimed-W}) yields that
  \begin{multline}
    \label{eq:proof-tauW-in-Lambda-Omega}
    \tau_C(\Sigma)=
    \left[
    \begin{pmatrix}
      I & -\tilde\Lambda_{\overline{C},C}\\
      0 &  I-\tilde\Lambda_{C,C} 
    \end{pmatrix}
    \begin{pmatrix}
      I & 0 \\
      0 & I-U_{C,C}
    \end{pmatrix}
    \right]^{-T}\\
    \times
    \begin{pmatrix}
      I & 0\\
      0& \Delta_{C,C}
    \end{pmatrix}
    \left[
    \begin{pmatrix}
      I & -\tilde\Lambda_{\overline{C},C}\\
      0 &  I-\tilde\Lambda_{C,C} 
    \end{pmatrix}
    \begin{pmatrix}
      I & 0 \\
      0 & I-U_{C,C}
    \end{pmatrix}
    \right]^{-1}.
  \end{multline}
  Using~(\ref{eq:proof:tildeOmega-princ}), we get that
  \begin{equation}
    \label{eq:proof:piece-together-Omegatilde}
    \begin{pmatrix}
      I & 0 \\
      0 & I-U_{C,C}
    \end{pmatrix}^{-T}
    \begin{pmatrix}
      I & 0\\
      0& \Delta_{C,C}
    \end{pmatrix}
    \begin{pmatrix}
      I & 0 \\
      0 & I-U_{C,C}
    \end{pmatrix}^{-1}
    =
    \begin{pmatrix}
      I & 0\\
      0 & \tilde\Omega_{C,C}
    \end{pmatrix}.
  \end{equation}
  Recalling the definition of $\Delta_\Lambda$ and $\tilde\Lambda$,
  we have
  \begin{equation}
    \label{eq:proof:piece-together-Lambda}
    \begin{pmatrix}
      I & -\tilde\Lambda_{\overline{C},C}\\
      0 &  I-\tilde\Lambda_{C,C} 
    \end{pmatrix}
    =
    \begin{pmatrix}
      I & -\Lambda_{\overline{C},C}\\
      0 &  I-\Lambda_{C,C} 
    \end{pmatrix}
    \begin{pmatrix}
      I & 0\\
      0 & (\Delta_\Lambda^{-1})_{C,C}
    \end{pmatrix}
  \end{equation}
  Plugging~(\ref{eq:proof:piece-together-Omegatilde})
  and~(\ref{eq:proof:piece-together-Lambda})
  into~(\ref{eq:proof-tauW-in-Lambda-Omega}), 
  we obtain the claim from~(\ref{eq:proof-to-show}).

  The entries of the matrices $A$ and $\Delta$ in
  (\ref{eq:proof-chol}) are rational functions of $\Sigma$ that are
  defined on all of $\mathit{PD}_V$.  Hence, the same is true for the
  map $\tau_C$ defined in~(\ref{eq:proof:claimed-W}).  
  
  The value of $\tau_C(\Sigma)$ uniquely determines the matrices
  $A_{C,C}$, $A_{\overline{C},C}$, and $\Delta_{C,C}$  in
  (\ref{eq:proof:claimed-W}).  They are determined through a block LDL
  decomposition and, thus, rational functions of $\tau_C(\Sigma)$.
  Knowing the three matrices for all $C\in\mathcal{C}(G)$, we can form
  $A$ and $\Delta$ and recover $\Sigma$ using~(\ref{eq:proof-chol}).
  Hence, $\tau$ is invertible and the inverse is rational.
\end{proof}


\addtocontents{toc}{\vspace{-.25cm}}{}{}
\part{Parameter Identification}
\addtocontents{toc}{\vspace{.02cm}}{}{}

\section{Global Identifiability}
\label{sec:glob-ident}

This section discusses Question~\ref{qu:glob-id}, which asks for a
characterization of the mixed graphs $G=(V,D,B)$ for which the map
$\phi_G$ is injective.  In the statistical literature a model with
injective parametrization is also called \emph{globally
  identifiable}.  

\begin{example}
  If $G$ is the graph from Figure~\ref{fig:sur}, then $\phi_G$ is
  injective.  Indeed, the coefficients for the two directed edges of
  $G$ satisfy $\sigma_{11}\lambda_{12}=\sigma_{12}$ and
  $\sigma_{44}\lambda_{43}=\sigma_{34}$;
  recall~(\ref{eq:sur-global-id}).  Since every positive definite
  matrix $\Sigma=(\sigma_{ij})\in\mathbb{R}^{4\times 4}$ has
  $\sigma_{11},\sigma_{44}>0$, these two equations always have a
  unique solution.  Hence, all fibers
  $\mathcal{F}_G(\Lambda,\Omega)$ are singleton sets.  In contrast, if
  $G$ is the graph from Figure~\ref{fig:iv:mixed}, then only generic
  fibers are singleton sets and $\phi_G$ is not injective; recall
  Example~\ref{ex:iv-generic}.
\end{example}

Our first observation ties in with classical linear algebra. 

\begin{theorem}
  \label{thm:glob-id-dags}
  If $G=(V,{\color{MyBlue} D},{\color{MyRed} \emptyset})$ is an
  acyclic digraph, then $\phi_G$ is injective and has a rational
  inverse.
\end{theorem}

We give two proofs.  The first one emphasizes the connection to
Cholesky decomposition.

\begin{proof}[Proof A]
  Suppose $V=\{1,\dots,m\}$ is enumerated in reversed
  topological order such that $i\to j\in D$ implies that $j<i$.  Then
  $\Lambda$ is a strictly lower-triangular matrix, and
  $\phi_G(\Lambda,\Omega)$ has matrix inverse
  $(I - \Lambda) \Omega^{-1} (I - \Lambda)^T$.  This is the product of
  a unit lower-triangular matrix, a positive diagonal matrix and a
  unit upper-triangular matrix.  We may compute $I-\Lambda$ and
  $\Omega^{-1}$ by an LDL decomposition.
\end{proof}

The second proof emphasizes the graphical nature of the problem and
possible sparsity of $\Lambda$.  It shows more explicitly that the
inverse of $\phi_G$ is rational.

\begin{proof}[Proof B]
  Letting $\Sigma=(\sigma_{ij})=\phi_G(\Lambda,\Omega)$, we have that
  \begin{equation}
    \label{eq:dag:lambda-id-rel}
    \Sigma_{\pa(i),i} \;=\; \Sigma_{\pa(i),\pa(i)}\Lambda_{\pa(i),i}    
  \end{equation}
  because if $j\in\pa(i)$, then every trek from $j$ to $i$ ends with
  an edge $k\to i$ for $k\in\pa(i)$.  Indeed, a trek from $j$ to $i$
  for which this fails has to be a directed path from $i$ to $j$.
  Adding the edge $j\to i$ to this path would yield a directed cycle.
  Similarly, every nontrivial trek from $i$ to $i$ begins and ends
  with a directed edge whose tail is a parent of $i$.  Hence,
  \begin{equation}
    \label{eq:dag:omega-id-rel}
    \sigma_{ii}\;=\;\omega_{ii} +\Lambda_{\pa(i),i}^T
    \Sigma_{\pa(i),\pa(i)}\Lambda_{\pa(i),i}. 
  \end{equation}
  The matrix $\Sigma_{\pa(i),\pa(i)}$ is a principal submatrix of the
  positive definite matrix $\Sigma$ and, thus, invertible.  Therefore,
  \begin{align}
    \label{eq:dag:lambda-id}
    \Lambda_{\pa(i),i} &\;=\;
                         \left(\Sigma_{\pa(i),\pa(i)}\right)^{-1}
                         \Sigma_{\pa(i),i},\\ 
    \label{eq:dag:omega-id}
    \omega_{ii} &\;=\;
                  \sigma_{ii}-\Sigma_{i,\pa(i)}\left(\Sigma_{\pa(i),\pa(i)}\right)^{-1} 
                  \Sigma_{\pa(i),i}.
                  \qedhere
  \end{align}
\end{proof}

Proof B shows that the formula from~(\ref{eq:dag:lambda-id}) holds
more generally.  It merely needs to hold that every trek from
a node $j\in\pa(i)$ to $i$ ends with a directed edge with $i$ as its head, so
an edge of the form $k\to i$.  This holds for every node in the graph
if and only if the graph is \emph{ancestral} \cite{richardson:2002}.
A mixed graph is ancestral if the presence of a directed path from
node $i$ to node $j$ implies that $i\not=j$ and $i\bi j\notin B$.  An
ancestral graph is in particular acyclic.

The next easy lemma is crucial for the understanding of injectivity of
$\phi_G$.
  
\begin{lemma}
  \label{lem:subgraphs}
  If $\phi_G$ is injective and $H \subseteq G$ is a
  \emph{subgraph}, then $\phi_H$ is injective.
\end{lemma}
\begin{proof}
  Any subgraphs can be obtained by removing edges one at a time, and
  then removing isolated nodes.  If $H$ is obtained from $G=(V,D,B)$
  by removing the edge $i\to j$, then $\phi_H$ is the restriction of
  $\phi_G$ to the subset of matrices
  $\Lambda\in\mathbb{R}^D_\mathrm{reg}$ that have $\lambda_{ij}=0$.
  If we instead remove the edge $i\bi j$, then the restrictions is to
  matrices $\Omega\in\mathit{PD}(B)$ with $\omega_{ij}=0$.  If $H$ is
  obtained by removing the isolated node $i$, then
  \[
    \phi_G(\Lambda,\Omega) =
    \begin{pmatrix}
      \phi_H(\Lambda,\Omega) & 0\\
      0 & \omega_{ii}
    \end{pmatrix}.
  \]
  In either case non-injectivity of $\phi_H$ implies non-injectivity
  of $\phi_G$.
\end{proof}

\begin{cor}
  \label{cor:global-id-simple}
  If $\phi_G$ is injective, then $G$ is {simple}, that is, for
  any vertices $i\not= j$ at most one of the three edges $i \bi j$,
  $i\toblue j$ and $i \otblue j$ may appear in $G$.
\end{cor}
\begin{proof}
  If $G$ is not simple then it contains a subgraph $H$ with two nodes
  and two edges.  The map $\phi_H$ is infinite-to-one as it maps a
  four-dimensional domain into the three-dimensional set of symmetric
  $2\times 2$ matrices.  Now apply
  Lemma~\ref{lem:subgraphs}.
\end{proof}

\begin{theorem}
  If $\phi_G$ is injective, then $G$ is {acyclic}.
\end{theorem}
\begin{proof}[Proof sketch]
  By Lemma~\ref{lem:subgraphs}, we may restrict to studying directed
  cycles $1 \toblue 2\toblue \dots \toblue m \toblue 1$.  The case
  of $m=2$ is covered by Corollary~\ref{cor:global-id-simple}.  If
  $m\ge 3$, then it is possible to show that $\phi_G$ is generically
  $2$-to-$1$, that is, the fiber $\mathcal{F}_G(\Lambda,\Omega)$ is
  generically of size two \cite{drton:2011}.
\end{proof}

It remains to characterize injectivity for acyclic graphs $G=(V,D,B)$.
The following theorem shows that injectivity can be decided in
polynomial time by alternatingly decomposing the bidirected part
$(V,B)$ into connected components and removing sink nodes of the
directed part $(V,D)$.  

\begin{theorem}
  \label{thm:glob-id}
  Suppose $G$ is an acyclic mixed graph.  Then:    
  \begin{enumerate}
  \item[(a)] $\phi_G$ is injective if and only if $G$ does not contain
    a subgraph whose bidirected part is connected and whose
    directed part has a unique sink.
  \item[(b)] If $\phi_G$ is injective, then its inverse is rational
    and $\mathcal{M}_G$ smooth.
  \end{enumerate}
\end{theorem}

Figure~\ref{fig:not-injective} illustrates the characterization in
part (a) of the theorem.  The fact that $\phi_G$ is not injective if
the combinatorial condition in (a) fails can be shown by a
counterexample for the particular subgraph and then invoking
Lemma~\ref{lem:subgraphs}.  For a proof of the sufficient condition in
Theorem~\ref{thm:glob-id}, we refer the reader to \cite{drton:2011}.
We note that the sufficiency of the condition for injectivity can be
proven by repeatedly applying the graph decomposition result in
Theorem~\ref{thm:tian-decomposition} and the result on ancestral
subgraphs from Theorem~\ref{thm:ancestral-subgraphs}.  These results
as well as Theorem~\ref{thm:glob-id} have generalizations to nonlinear
structural equation models \cite{shpitser:2006,tian:pearl:2002}.

  \begin{figure}[t]
    \centering
    \begin{tabular}{ll}
      (a) \scalebox{0.9}{
        \begin{tikzpicture}[->,>=triangle 45, shorten >=0pt,
          auto,thick, main node/.style={circle,inner
            sep=2pt,fill=gray!20,draw,font=\sffamily}]
          
          \node[main node,rounded corners] (1) {1}; \node[main
          node,rounded corners] (2) [below right=.75cm and .75cm of 1]
          {2}; \node[main node,rounded corners] (3) [right=2.5cm of 1]
          {3}; \node[main node,rounded corners] (4) [below
          right=0.75cm and .75cm of 3] {4};
          
          \path[color=black!20!blue,every
          node/.style={font=\sffamily\small}] (1) edge node {} (2) (1)
          edge node {} (4) (2) edge node {} (3) (3) edge node {} (4);
          \path[color=black!20!red,<->,every
          node/.style={font=\sffamily\small}] (1) edge node {} (3) (2)
          edge node {} (4);
        \end{tikzpicture}
      } 
      &
        \qquad (b) \scalebox{0.9}{
            \begin{tikzpicture}[->,>=triangle 45, shorten >=0pt,
              auto,thick, main node/.style={circle,inner
                sep=2pt,fill=gray!20,draw,font=\sffamily}]
          
              \node[main node,rounded corners] (1) {1}; \node[main
              node,rounded corners] (2) [below right=.75cm and .75cm
              of 1] {2}; \node[main node,rounded corners] (3)
              [right=2.5cm of 1] {3}; \node[main node,rounded corners]
              (4) [below right=0.75cm and .75cm of 3] {4};
          
              \path[color=black!20!blue,every
              node/.style={font=\sffamily\small}] (1) edge node {} (2)
              (2) edge node {} (3) (3) edge node {} (4);
              \path[color=black!20!red,<->,every
              node/.style={font=\sffamily\small}] (1) edge node {} (3)
              (1) edge node {} (4) (2) edge node {} (4);
            \end{tikzpicture}
          }      
    \end{tabular}
        \caption{(a) A mixed graph for which the parametrization $\phi_G$
          is injective.  (b) A graph for which $\phi_G$ is not injective.}
        \label{fig:not-injective}
      \end{figure}
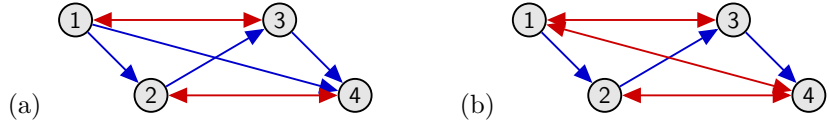

\section{Generic Identifiability}
\label{sec:gener-ident}

The difference between injectivity and generic injectivity of $\phi_G$
may appear minute.  However, the two properties are quite different,
and failure of generic injectivity cannot be argued by studying
subgraphs (as in Lemma~\ref{lem:subgraphs}).  According to
Corollary~\ref{cor:global-id-simple}, a mixed graph $G$ can have the
map $\phi_G$ injective only if it is acyclic and simple.  The deeper
issue is then to find out which simple acyclic mixed graphs have
$\phi_G$ injective.  In contrast, the next result shows that all
simple acyclic mixed graphs are generically injective.  The deeper
issue for generic injectivity is thus the treatment of graphs that
contain directed cycles or are not simple.

\begin{theorem}
  \label{thm:simple-gen-id}
  If $G=(V,D,B)$ is acyclic and simple, then $\phi_G$ is generically
  injective and algebraically one-to-one.
\end{theorem}

The theorem is due to \cite{brito:2002}.  It shows that the graph from
Figure~\ref{fig:not-injective}(b) has $\phi_G$ generically injective,
but not injective.  A short proof of Theorem~\ref{thm:simple-gen-id}
is obtained through the following observation that we use repeatedly
in this section.

\begin{lemma}
  \label{lem:gen-id-easier-equations}
  Let $G=(V,D,B)$ be a mixed graph, and let
  $\Sigma=\phi_G(\Lambda_0,\Omega_0)$ for
  $\Lambda_0\in\mathbb{R}^D_\mathrm{reg}$ and
  $\Omega_0\in\mathit{PD}(B)$.  The fiber
  $\mathcal{F}_G(\Lambda_0,\Omega_0)$ is isomorphic to the set of
  matrices $\Lambda\in\mathbb{R}^D_\mathrm{reg}$ that solve the
  equation system
  \begin{equation}
    \label{eq:gen-id-easier-equations}
    \left[(I-\Lambda)^T\Sigma(I-\Lambda)\right]_{ij} \;=\; 0, \qquad
    i\not=j, \; i\bi j\notin B.
  \end{equation}
\end{lemma}
\begin{proof}
  The projection $(\Lambda,\Omega)\mapsto\Lambda$ maps
  $\mathcal{F}_G(\Lambda_0,\Omega_0)$ to the set of matrices
  $\Lambda\in\mathbb{R}^D_\mathrm{reg}$ that solve the equations
  in~(\ref{eq:gen-id-easier-equations}).  Indeed, as $I-\Lambda$ is
  invertible for $\Lambda\in\mathbb{R}^D_\mathrm{reg}$,
  \[      \Sigma=\phi_G(\Lambda,\Omega)=(I-\Lambda)^{-T}\Omega(I-\Lambda) 
    \;\implies\;
    \Omega=(I-\Lambda)^T\Sigma(I-\Lambda).
  \]
  Any entry of $\Omega$ that is indexed by a pair $(i,j)$ with
  $i\not=j$ and $i\bi j\notin B$ is zero.  Conversely, if
  $\Lambda\in\mathcal{F}_G(\Lambda)$, then
  $(\Lambda,(I-\Lambda)^T\Sigma(I-\Lambda))\in\mathcal{F}_G(\Lambda,\Omega)$.
\end{proof}

We emphasize that the equations in~(\ref{eq:gen-id-easier-equations})
are bilinear as
\begin{equation*}
  \left[(I-\Lambda)^T\Sigma(I-\Lambda)\right]_{ij} 
  \;=\;
  \sigma_{ij} - \sum_{k\in\pa(i)} \lambda_{ki}\sigma_{ki} -
  \sum_{l\in\pa(j)} \sigma_{il}\lambda_{lj} + \sum_{k\in\pa(i)}\sum_{l\in\pa(j)} \lambda_{ki}\sigma_{kl}\lambda_{lj}.
\end{equation*}

\begin{proof}[Proof of Theorem~\ref{thm:simple-gen-id}]
  Because $G$ is acyclic, we may enumerate the vertex set in a
  topological order as $V=\{1,\dots,m\}$.  Then
  $\pa(i)\subseteq\{1,\dots,i-1\}$ for $i=1,\dots,m$.  Moreover,
  because $G$ is simple, $j\in\pa(i)$ implies that $j\bi i\notin B$.
  By Lemma~\ref{lem:gen-id-easier-equations},
  \[
    \left[(I-\Lambda)^T\Sigma(I-\Lambda)\right]_{\pa(i),i} \;=\; 0,
    \qquad i=1,\dots,m.
  \]
  These equations can be rewritten as
  \begin{equation}
    \label{eq:acyclic-simple-recursive}
    \left[(I-\Lambda)^T\Sigma\right]_{\pa(i),\pa(i)}
    \Lambda_{\pa(i),i} \;=\; \left[(I-\Lambda)^T\Sigma\right]_{\pa(i),i},
    \qquad i=1,\dots,m.
  \end{equation}
  By the tolopogical order, if $j\in\pa(i)$, then the entries in the
  $j$-th row of $(I-\Lambda)^T\Sigma$ depend only on the first $i-1$
  columns of $\Lambda$.  The system
  in~(\ref{eq:acyclic-simple-recursive}) can thus be solved
  recursively, where each step requires solving a linear system.

  To show that $\phi_G$ is generically injective, it remains to argue
  that the equations in~(\ref{eq:acyclic-simple-recursive})
  generically have a unique solution.  It suffices to exhibit a single
  pair $(\Lambda,\Omega)$ for which this is true.  We may choose
  $\Lambda=0$ and $\Omega=I$, so $\Sigma=I$.  Then the
  matrix for the $i$-th group of equations
  in~(\ref{eq:acyclic-simple-recursive}) is $\Sigma_{\pa(i),\pa(i)}$,
  which is invertible.
\end{proof}

Although a combinatorial characterization of the graphs with
generically injective parametrization $\phi_G$ is not known, Gr\"obner
basis techniques can be used to determine the degree of
identifiability from Definition~\ref{def:id-degree} and, thus, decide
whether $\phi_G$ is algebraically one-to-one.  Gr\"obner bases are
computationally tractable for non-trivial examples and have been used
for a classification of all graphs with up to 5 nodes
\cite{foygel:draisma:drton:2012}.  For larger graphs, algebraic
methods can be applied after decomposition according to
Theorem~\ref{thm:tian-decomposition}.  

We describe two options for the computation.  In either case, we
advocate working with the equation system
from~(\ref{eq:gen-id-easier-equations}) as opposed to the fiber
equation $\Sigma=\phi_G(\Lambda,\Omega)$.
System~(\ref{eq:gen-id-easier-equations}) has $\Omega$ eliminated and
may be far more compact as it avoids the inversion of $I-\Lambda$.
This said, although system~(\ref{eq:gen-id-easier-equations}) is
polynomial also for graphs that contain directed cycles, care must be
taken to avoid spurios solutions with $I-\Lambda$ non-invertible.

The first possibility is to perform a parametric Gr\"obner basis
computation.  We introduce a matrix $\Lambda$ whose nonzero entries
$\lambda_{ij}$, $i\to j\in D$, are indeterminates and a pair of
matrices $(\Lambda_0,\Omega_0)$ that are parameters.  We form the
matrix $(I-\Lambda)^T\phi_G(\Lambda_0,\Omega_0)(I-\Lambda)$ and set to
zero the off-diagonal entries indexed by non-edges of the bidirected
part $(V,B)$.  We then compute a Gr\"obner basis for the resulting
system in the polynomial ring with coefficients in the field of
rational fractions $\mathbb{R}(\Lambda_0,\Omega_0)$.  The Gr\"obner
basis readily yields the dimension of the generic fibers.  If the
dimension is finite we may also find the algebraic degree of the
generic fibers, which is what we referred to as degree of
identifiability.  When the graph $G$ contains directed cycles, the
matrix $I-\Lambda$ can be non-invertible.  We thus first saturate our
equation system with respect to $\det(I-\Lambda)$.

\begin{example}
  The following code for the system {\sc Singular} \cite{Singular}
  implements the approach just described for a directed 3-cycle:
  \begin{small}
\begin{verbatim}
LIB "linalg.lib"; option(redSB);
ring R = (0,l012,l023,l031,w011,w022,w033),(l12,l23,l31),dp;
matrix L[3][3] = 1,-l12,0,  
                 0,1,-l23,
                 -l31,0,1;
matrix L0[3][3] = 1,-l012,0,
                  0,1,-l023,
                  -l031,0,1;
matrix W0[3][3] = w011,0,0,
                  0,w022,0,
                  0,0,w033;
matrix W[3][3] = transpose(L)*inverse(transpose(L0))*W0*inverse(L0)*L;
ideal GB = sat(ideal(W[1,2],W[1,3],W[2,3]), det(L))[1];
dim(GB); mult(GB);
\end{verbatim}
  \end{small}
  The output from the last line first certifies that the fibers are
  generically zero-dimensional, that is, contain finitely many
  points.  The multiplicity computed with the last command shows the
  degree of identifiability to be two.
\end{example}

The second possibility is to consider only polynomials with
real-valued coefficients but to introduce as polynomial variables the
nonzero entries of $\Lambda$ as well as a symmetric matrix $\Sigma$.
These variables are ordered with respect to a block monomial order in
which the variables in $\Lambda$ are larger than the variables in
$\Sigma$.  Let $\mathcal{I}$ be the ideal generated by the
off-diagonal entries of $(I-\Lambda)^T\Sigma(I-\Lambda)$ that are
indexed by the non-edges of $(V,B)$.  Saturate $\mathcal{I}$ with
respect to $\det(I-\Lambda)$.  Let $\mathcal{J}$ be the reduced
Gr\"obner basis of the resulting ideal.  Elimination theory yields the
following \cite[Section 8 of the supplemental material]{foygel:draisma:drton:2012}.

\begin{proposition}
  \label{prop:leading-terms}
  The parametrization $\phi_G$ of a mixed graph $G=(V,D,B)$ is
  algebraically one-to-one if and only if for each $i\to j\in D$, the
  reduced Gr\"obner basis $\mathcal{J}$ contains an element with
  leading monomial $a(\Sigma)\lambda_{ij}$.
\end{proposition}

More generally, the generic dimension and degree of the fibers
$\mathcal{F}_G(\Lambda,\Omega)$ can be determined by analyzing the
monomials under the staircase of the initial ideal of $\mathcal{J}$
\cite[Chapter 9]{cox:2007}.  This way we may determine the degree of
identifiability of $G$.  In comparison to the first approach, the
second method yields relations that show how to identify coefficients
$\lambda_{ij}$ from $\Sigma$.

\begin{example}
  Treating again a directed 3-cycle, we give an example of the second
  type of computation in {\sc Singular}:
  \begin{small}
\begin{verbatim}
LIB "linalg.lib"; option(redSB);
ring R = 0,(l12,l23,l31,s11,s12,s13,s22,s23,s33),(dp(3));
matrix L[3][3] = 1,-l12,0,
                 0,1,-l23,
                 -l31,0,1;
matrix S[3][3] = s11,s12,s13,
                 s12,s22,s23,
                 s13,s23,s33;
matrix W[3][3] = transpose(L)*S*L;
ideal GB = sat(ideal(W[1,2],W[1,3],W[2,3]), det(L))[1]; GB;
\end{verbatim}
  \end{small}
  The output is a list of 9 polynomials whose leading terms are, in
  our usual notation, 
  \begin{align*}
    \lambda_{23}\lambda_{31}\sigma_{23}, &&
    \lambda_{12}\lambda_{31}\sigma_{13}, &&
    \lambda_{12}\lambda_{23}\sigma_{12}, &&
    \lambda_{12}\lambda_{31}\sigma_{12}\sigma_{33},\\
    \lambda_{12}\sigma_{11}\sigma_{13}\sigma_{23}, &&
    \lambda_{12}\lambda_{23}\sigma_{11}\sigma_{23}, &&
    \lambda_{23}\sigma_{12}\sigma_{13}\sigma_{22}, &&
    \lambda_{23}\lambda_{31}\sigma_{13}\sigma_{22}, &&
    \lambda_{31}^2\sigma_{13}\sigma_{22}\sigma_{33}.
  \end{align*}
  By Proposition~\ref{prop:leading-terms}, $\phi_G$ is not
  algebraically one-to-one because there is no leading term of the
  form $\lambda_{31}a(\Sigma)$.  The last leading term belongs to a
  polynomial that shows that $\lambda_{31}$ is algebraic function of
  degree 2 of the covariance matrix $\Sigma$ because it solves the
  equation
  \begin{multline*}
    \lambda_{31}^2\sigma_{33}(\sigma_{13}\sigma_{22}-\sigma_{12}\sigma_{23})
    -\lambda_{31}(\sigma_{13}^2\sigma_{22}-\sigma_{11}\sigma_{23}^2-\sigma_{12}^2\sigma_{33}+\sigma_{11}\sigma_{22}\sigma_{33})\\
    +\sigma_{11}(\sigma_{13}\sigma_{22}-\sigma_{12}\sigma_{23}) 
    \;=\;0.
  \end{multline*}
  The equations with leading terms
  $\lambda_{23}\sigma_{12}\sigma_{13}\sigma_{22}$ and
  $\lambda_{12}\sigma_{11}\sigma_{13}\sigma_{23}$ show that
  $\lambda_{23}$ and $\lambda_{12}$ are rational functions of $\Sigma$
  and $\lambda_{31}$.  Altogether, we have verified that $\phi_G$ is
  algebraically 2-to-one.  Checking this by counting monomials under
  the staircase means considering the leading monomials considering
  only the variables $\lambda_{12},\lambda_{23},\lambda_{31}$ we seek
  to solve for.  The monomials are
  \begin{align}
    \label{eq:9monomials}
    \lambda_{23}\lambda_{31}, &&
    \lambda_{12}\lambda_{31}, &&
    \lambda_{12}\lambda_{23}, &&
    \lambda_{12}, &&
    \lambda_{23}, &&
    \lambda_{31}^2,
  \end{align}
  and generate the ideal
  $\mathcal{I}=\langle
  \lambda_{12},\lambda_{23},\lambda_{31}^2\rangle$.  The monomials
  under the staircase are the monomials in
  $\mathbb{R}[\lambda_{12},\lambda_{23},\lambda_{31}]\setminus
  \mathcal{I}$.  There are two, namely, $1$
  and $\lambda_{31}$.
%
\end{example}

Although Gr\"obner basis methods can be effective, it is desirable to
obtain combinatorial methods that are efficient also for large-scale
problems.  The half-trek criteria of \cite{foygel:draisma:drton:2012}
are state-of-the-art criteria that can be checked in time that is
polynomial in the size of the vertex set of the considered graph.
They provide a sufficient as well as a necessary condition for generic
injectivity of $\phi_G$.  More precisely, there is a condition that
is sufficient for $\phi_G$ to be algebraically one-to-one and a
related condition that is neessary for $\phi_G$ to be generically
finite-to-one.  The conditions are implemented in a package for the R
project for statistical computing \cite{semid}.  We begin our
discussion of the half-trek criteria by introducing some needed
terminology.

A \emph{half-trek} from source node $i$ to target node $j$ is a trek $\tau$
from $i$ to $j$ whose left-hand side is a singleton set, so
$\lhs{\tau}=\{i\}$.  In other words, a half-trek is of the form
\[
  i \toblue j_1 \toblue\dots\toblue j_r \to j \quad\text{or}\quad i \bi
  j_1 \toblue\dots\toblue j_r \to j.
\]
Let $X,Y\subseteq V$ be two sets of nodes of equal cardinality
$|X|=|Y|=k$.  Let $\Pi$ be a set of $k$ treks.  Then $\Pi$ is a system
of treks from $X$ to $Y$, denoted $\Pi:X\rightrightarrows Y$, if $X$
is the set of source nodes of the treks in $\Pi$ and $Y$ is the set of
target nodes.  Note that we allow $X\cap Y\not=\emptyset$.  The system
$\Pi$ is a system of half-treks if every trek $\pi_i$ is a half-trek.
Finally, a system $\Pi$ has no sided intersection if
\[
\lhs{\pi}\cap\lhs{\pi'}=\emptyset=\rhs{\pi}\cap\rhs{\pi'}
\]
for all pairs of treks $\pi,\pi'\in\Pi$.

\begin{definition}
  \label{def:htc}
  A set $Y\subseteq V$ satisfies the \emph{half-trek criterion} with
  respect to node $i$ if (i) $|Y|=|\pa(i)|$, (ii) $j=i$ or $j\bi i$
  implies that $j\not\in Y$, and (iii) there exists a system of
  half-treks $\Pi:Y\rightrightarrows\pa(i)$ that has no sided
  intersection.
\end{definition}

\begin{theorem} \label{thm:htc-id} 
  Let $G=(V,D,B)$ be a mixed graph.
  \begin{enumerate}
  \item[(i)] If for every $i\in V$ there exists a set $Y_i\subseteq V$
    that satisfies the half-trek criterion with respect to $i$, and if
    there exists a total ordering $\prec$ such that $j\prec i$
    whenever $j\in Y_i$ and there is a half-trek from $i$ to $j$, then
    $\phi_G$ is generically injective and algebraically one-to-one.
  \item[(ii)] For $\phi_G$ to be generically finite-to-one it is
    necessary that there exists a family of subsets $Y_i\subseteq V$,
    $i\in V$, such that $Y_i$ satisfies the half-trek criterion with
    respect to $i$ and $j\in Y_i$ implies $i\not \in Y_j$.
  \end{enumerate}
\end{theorem}

We merely outline the proof of the theorem; for details see
\cite{foygel:draisma:drton:2012}.  Some of the arguments are further
illustrated in Example~\ref{ex:htc-ideas}.  Note also that
Theorem~\ref{thm:simple-gen-id} is obtained from
Theorem~\ref{thm:htc-id}(i) by taking $Y_i=\pa(i)$ and $\prec$ as a
topological order.

\begin{proof}[Outline of proof of Theorem~\ref{thm:htc-id}]
  (i) Let $\Sigma=\phi_G(\Lambda_0,\Omega_0)$ for
  $\Lambda_0\in\mathbb{R}^D_\mathrm{reg}$ and
  $\Omega_0\in\mathit{PD}(B)$.  Suppose
  $(\Lambda,\Omega)\in\mathcal{F}_G(\Lambda_0,\Omega_0)$.  To show
  that $(\Lambda,\Omega)=(\Lambda_0,\Omega_0)$, we visit the nodes
  $i\in V$ from smallest to largest in the order $\prec$ and
  iteratively find a linear equation system that is uniquely solved
  by the $i$-th column of $\Lambda$.  The starting point is
  Lemma~\ref{lem:gen-id-easier-equations}, by which we have
  \begin{equation}
    \label{eq:htc-equations}
    \left[(I-\Lambda)^T\Sigma(I-\Lambda)\right]_{Y_i,i} = 0,  \quad
  i\in V.
  \end{equation}
  This is true because Definition~\ref{def:htc} yields that $j\not=i$
  and $j\bi i\notin B$ when $j\in Y_i$.  Similar to the proof of
  Theorem~\ref{thm:simple-gen-id}, we may
  rearrange~(\ref{eq:htc-equations}) to 
  \[
    A_i(\Lambda,\Sigma) \Lambda_{\pa(i),i} \;=\; b_i(\Lambda,\Sigma),
  \]
  with
  $A_i(\Lambda,\Sigma)=\left[(I-\Lambda)^T\Sigma\right]_{Y_i,\pa(i)}$
  and $b_i(\Lambda,\Sigma)= \left[(I-\Lambda)^T\Sigma\right]_{Y_i,i}$.
  Both $A_i(\Lambda,\Sigma)$ and $b_i(\Lambda,\Sigma)$ can be shown to
  only depend on those columns of $\Lambda$ that are indexed by nodes
  $j$ with a half-trek from $i$ to $j$.  Hence, the proof is complete
  if we can show that $A_i(\Lambda_0,\Sigma)$ is invertible for
  generic choices of $\Lambda_0$ and $\Omega_0$.  To verify this, we
  may use the existence of a half-trek system without sided
  intersection from $Y_i$ to $\pa(i)$ to argue that the determinant of
  $A_i(\Lambda_0,\Sigma)$ is not the zero polynomial.  This last step
  is in the spirit of the Lindstr\"om-Gessel-Viennot lemma.
  
  (ii) We study the Jacobian of the equations from
  Lemma~\ref{lem:gen-id-easier-equations}.  Its rows are indexed by
  the non-edges of the bidirected part $(V,B)$ and its columns by the
  edges in $D$.  For $\phi_G$ to be generically finite-to-one, it is
  necessary that the Jacobian has full column rank $D$.  It can be
  shown that the Jacobian contains an invertible $|D|\times |D|$
  submatrix only if the given condition holds.  Let $J_i$ be the
  submatrix of the Jacobian obtained by selecting the columns
  corresponding to directed edges with head at $i$.  Then, more specifically,
  $J_i$ has full column rank only if there exists a subset
  $Y_i\subseteq V$ that satisfies the half-trek criterion with respect
  to $i$.  Moreover, if $j\in Y_i$ and $i \in Y_j$, then the same row,
  namely, that corresponding to $i\bi j\notin B$, would be used to get
  an invertible square submatrix of $J_i$ and $J_j$.
\end{proof}

The conditions from Theorem~\ref{thm:htc-id} seem involved but as we
mentioned they can be checked in polynomial time.  Indeed, the problem
of finding suitable sets $Y_i$ that satisfy the half-trek criteria can
be shown to correspond to solving network-flow problems.  For
condition (i), we repeatedly solve a network-flow problem.  Condition
(ii) can be implemented as a single larger network-flow problem
\cite[Section 6]{foygel:draisma:drton:2012}.

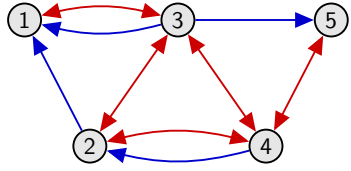
\begin{figure}[t]
  \centering
    \scalebox{0.9}{
    \begin{tikzpicture}[->,>=triangle 45, shorten >=0pt,
      auto,thick, main node/.style={circle,inner
        sep=2pt,fill=gray!20,draw,font=\sffamily}]
      
      \node[main node] (1) {1}; 
      \node[main node] (3) [right=1.75cm of 1] {3}; 
      \node[main node] (5) [right=1.75cm of 3] {5}; 
      \node[main node] (2) [below right=1.5cm and 0.6cm of 1] {2};
      \node[main node] (4) [below left=1.5cm and 0.6cm of 5] {4};
      
      \path[color=black!20!blue] 
      (2) edge node {} (1) 
      (3) edge node {} (5);

      \path[color=black!20!blue, bend left=15]
      (4) edge node {} (2)
      (3) edge node {} (1);

      \path[color=black!20!red,<->, bend left=15] 
      (2) edge node {} (4)
      (1) edge node {} (3);
      \path[color=black!20!red,<->] 
      (2) edge node {} (3)
      (4) edge node {} (3)
      (4) edge node {} (5);
    \end{tikzpicture}
  }
  \caption{Illustration of Theorem~\ref{thm:htc-id}.}
  \label{fig:htc-id}
\end{figure}

\begin{example}
  \label{ex:htc-ideas}
  Let $G$ be the graph in Figure~\ref{fig:htc-id}.   Each one of the sets
  \begin{align*}
    Y_1&=\{2,5\}, & Y_2&=\{5\}, & Y_3&=\emptyset, & Y_4&=\emptyset, & Y_5&=\{3\}
  \end{align*}
  satisfies the half-trek criterion with respect to the node it is
  indexed by.  This is least evident for $Y_1$, and we highlight the
  half-treks $2 \bi 3$ and $5\bi 4\to 2$, which have no sided
  intersection.  Choosing the ordering as
  $3\prec 4 \prec 5 \prec 1\prec 2$, Theorem~\ref{thm:htc-id}(i)
  shows that $\phi_G$ is algebraically one-to-one.  Other possible
  orderings are obtained by permuting $\{3,4,5\}$ or $\{1,2\}$.

  To illustrate ideas from the proof of Theorem~\ref{thm:htc-id}(i),
  we focus on node 1, with $\pa(1)=\{2,3\}$.  Since $Y_1=\{2,5\}$, we
  work with the equations
  \begin{align*}
    \left[(I-\Lambda)^T\Sigma(I-\Lambda)\right]_{51}&\;=\;0, &
    \left[(I-\Lambda)^T\Sigma(I-\Lambda)\right]_{21}&\;=\;0.
  \end{align*}
  Expanding out the matrix product, the equations become
  \begin{align}
    \label{eq:ex-htc-id-51}
    \sigma_{51}-\left(
    \sigma_{52}{\lambda_{21}}+
    \sigma_{53}{\lambda_{31}}\right)-\lambda_{35}\sigma_{31}+\left(
    \lambda_{35}\sigma_{32}{\lambda_{21}}+
    \lambda_{35}\sigma_{33}{\lambda_{31}}\right)&\;=\;0,\\
    \label{eq:ex-htc-id-21}
    \sigma_{21}-\left(
    \sigma_{22}{\lambda_{21}}+
    \sigma_{23}{\lambda_{31}}\right)-
    \lambda_{42}\sigma_{41}+\left(
    \lambda_{42}\sigma_{42}{\lambda_{21}}+
    \lambda_{42}\sigma_{43}{\lambda_{31}}\right)&\;=\;0,
  \end{align}
  and we wish to solve for $\lambda_{21}$ and $\lambda_{31}$.  With
  $5\prec 1$, we have already solved for $\lambda_{35}$; since
  $Y_5=\pa(5)$ we have $\lambda_{35}=\sigma_{35}/\sigma_{55}$ as is
  also clear from the discussion after Proof B for
  Theorem~\ref{thm:glob-id-dags}.  Substituting the ratio for
  $\lambda_{35}$ turns~(\ref{eq:ex-htc-id-51}) into a linear equation
  in $\lambda_{21}$ and $\lambda_{31}$.  The equation
  in~(\ref{eq:ex-htc-id-21}) could be linearized similarly, except
  that now the relevant coefficient $\lambda_{42}$ has not yet been
  determined in an ordering with $1\prec 2$.  However, if $\Lambda$ is
  part of a pair $(\Lambda,\Omega)$ in the fiber given by $\Sigma$,
  then
  \[
    -
    \lambda_{42}\sigma_{41}+\left(
    \lambda_{42}\sigma_{42}{\lambda_{21}}+
    \lambda_{42}\sigma_{43}{\lambda_{31}}\right)\;=\;0
  \]
  because there is no half-trek from $1$ to $2$.  To see this note
  that the term $\lambda_{42}\sigma_{41}$ corresponds to treks from
  $2$ to $1$ that start with the edge $2\leftarrow 4$, whereas the sum
  $\lambda_{42}\sigma_{42}{\lambda_{21}}+
  \lambda_{42}\sigma_{43}{\lambda_{31}}$ corresponds to
  treks from $2$ to $1$ that start with the edge $2\leftarrow 4$ and
  end in either $2\to 1$ or $3\to 1$.  These two sets of treks
  coincide when there is no half-trek from $1$ to $2$.  
\end{example}

The sufficient condition from Theorem~\ref{thm:htc-id}(i) can be
strengthened by first applying the graph decomposition from
Section~\ref{sec:graph-decomposition} and then check condition (i) in
each subgraph.  No such strengthening is possible for the necessary
condition from part (ii) of the theorem
\cite{foygel:draisma:drton:2012}.  Further strengthening of the
sufficient condition is possible by first removing sink nodes from the
graph using the observation from
Theorem~\ref{thm:ancestral-subgraphs}.  When a sink node is removed a
more refined graph decomposition may become possible; we refer the
reader to \cite{chen:2015,drton:weihs:2015}.  While a specific
polynomial-time algorithm using this idea is given in
\cite{drton:weihs:2015}, it is still unclear how to best design
algorithms based on recursive graph decomposition and removal of sink
nodes.

We conclude our discussion of parameter identification with two
examples from the exhaustive computational study of graphs with up to
5 nodes in \cite{foygel:draisma:drton:2012}.  Both graphs in
Figure~\ref{fig:htc-id} satisfy the necessary condition in
Theorem~\ref{thm:htc-id}(ii) and, thus, have $\phi_G$ generically
finite-to-one.  Neither graph satisfies the sufficient condition from
Theorem~\ref{thm:htc-id}(i).  The graph in panel (a) indeed does not
have $\phi_G$ generically injective.  Instead, $\phi_G$ is
algebraically 3-to-one.  The graph in panel (b), however, is
algebraically one-to-one but Theorem~\ref{thm:htc-id}(i) fails to
recognize it.  Decomposition and removal of sink nodes do not help.


\begin{figure}[t]\centering
  \begin{tabular}{p{6.25cm}@{\hspace{1.75cm}}p{3.5cm}}
    (a)
      \scalebox{0.9}{
      \begin{tikzpicture}[->,>=triangle 45, shorten >=0pt,
        auto,thick, main node/.style={circle,inner
          sep=2pt,fill=gray!20,draw,font=\sffamily}]
        
        \node[main node,rounded corners] (1) {1}; \node[main
          node,rounded corners] (2) [right=1.5cm of 1] {2}; \node[main
          node,rounded corners] (3) [below right=.5cm and 1.5cm of 2]
          {3}; \node[main node,rounded corners] (4) [below right=.5cm
          and 1.5cm of 3] {4}; \node[main node,rounded corners] (5)
          [above right=1cm and 1.5cm of 2] {5};
          
          \path[color=black!20!blue,every
          node/.style={font=\sffamily\small}] (1) edge node {} (2) (2)
          edge node {} (3) (3) edge node {} (4) (1) edge node {} (5)
          (2) edge node {} (5) ; \path[color=black!20!blue,every
          node/.style={font=\sffamily\small},bend right] (1) edge node
          {} (4); \path[color=black!20!red,<->,every
          node/.style={font=\sffamily\small},bend right] (1) edge node
          {} (2) (1) edge node {} (3) ;
          \path[color=black!20!red,<->,every
          node/.style={font=\sffamily\small},bend left] (1) edge node
          {} (5) (2) edge node {} (4);
        \end{tikzpicture}
      }
    & 
      (b)
    \scalebox{0.9}{
    \begin{tikzpicture}[->,>=triangle 45, shorten >=0pt,
      auto,thick, main node/.style={circle,inner
        sep=2pt,fill=gray!20,draw,font=\sffamily}]
      
      \node[main node,rounded corners] (1) {1}; \node[main
      node,rounded corners] (2) [below=1.5cm of 1] {2}; \node[main
      node,rounded corners] (3) [right=1.5cm of 2] {3}; \node[main
      node,rounded corners] (4) [right=1.5cm of 1] {4}; \node[main
      node,rounded corners] (5) [above right=1cm and .6cm of 1]
      {5};
      
      \path[color=black!20!blue,every
      node/.style={font=\sffamily\small}] (1) edge node {} (2) (1)
      edge node {} (3) (1) edge node {} (4) (4) edge node {} (5) ;
      \path[color=black!20!red,<->,every
      node/.style={font=\sffamily\small},bend right] (1) edge node
      {} (2) (1) edge node {} (3) (1) edge node {} (4);
      \path[color=black!20!red,<->,every
      node/.style={font=\sffamily\small}] (1) edge node {} (5);
    \end{tikzpicture}
    }
  \end{tabular}
  \caption{Graphs that satisfy the necessary but not the
    sufficient condition from Theorem~\ref{thm:htc-id}: (a) $\phi_G$
    is algebraically 3-to-one, (b) $\phi_G$ is algebraically
    one-to-one and, thus, generically injective.}
  \label{fig:htc-gap}
\end{figure}
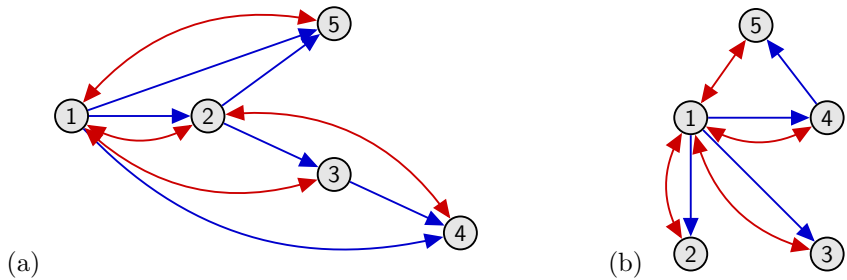

\addtocontents{toc}{\vspace{-.25cm}}{}{}
\part{Relations Among Covariances}
\addtocontents{toc}{\vspace{.02cm}}{}{}

\section{Implicitization}
\label{sec:implicitization}

Let $\mathcal{M}_G$ be the set of covariance matrices of the
structural equation model given by a mixed graph $G=(V,D,B)$.
Motivated, in particular, by the covariance equivalence problem from
Question~\ref{qu:model-equivalence}, we now discuss polynomial
relations among the entries of the matrices in $\mathcal{M}_G$.  Let
$\Sigma=(\sigma_{ij})$ be a symmetric $V\times V$ matrix of variables,
and let $\mathbb{R}[\Sigma]$ be the ring of polynomials in the
$\sigma_{ij}$ with real coefficients.  Then the polynomial relations
we seek to understand make up the \emph{vanishing ideal}
\[
  \mathcal{I}(G) \;=\: \left\{
    f\in \mathbb{R}[{\Sigma}] \::\:
    f({\Sigma}) = 0\,\;\text{for all}\;\, 
    {\Sigma}\in\mathcal{M}_G 
  \right\}.      
\]
Since $\mathcal{M}_G$ is the image of the rational map $\phi_G$, we
may compute a generating set for $\mathcal{I}(G)$ by implicitization.
Assuming for simplicity that $G$ is acyclic and, thus, $\phi_G$
polynomial, we have
\[
  \mathcal{I}(G) \;=\; \langle\; {\Sigma}-\phi_G({\Lambda},{\Omega}) \;\rangle
  \;\cap\; \mathbb{R}[{\Sigma}].
\]
A better approach, however, is to start with the equations from
Lemma~\ref{lem:gen-id-easier-equations}, which have $\Omega$ already
eliminated.  We compute
\[
  \mathcal{I}(G) \;=\; \left\langle
    \left[(I-{\Lambda})^\trans{\Sigma}(I-{\Lambda}) 
    \right]_{ij}\::\:\; i \neq j,\; i \bi j \notin B \right\rangle
  \;\cap\; \mathbb{R}[{\Sigma}].
\]
If $G$ is not acyclic, we saturate with respect to
$\det(I-{\Lambda})$ before intersecting with
$\mathbb{R}[{\Sigma}]$; compare the examples in
Section~\ref{sec:gener-ident}.

\begin{figure}[t]
  \centering
    \scalebox{0.9}{
      \begin{tikzpicture}[->,>=triangle 45, shorten >=0pt, auto,thick,
        main node/.style={circle,inner
          sep=2pt,fill=gray!20,draw,font=\sffamily}]
        
        \node[main node] (1) {1}; \node[main node] (2) [below
        right=0.5cm and 1.25cm of 1] {2}; \node[main node] (3) [right=2.5cm of 1] {3}; \node[main node] (4) [below
        right=0.5cm and 1.25cm of 3] {4};
        
        \path[color=black!20!blue,every
        node/.style={font=\sffamily\small}] (1) edge node {} (2) (1)
        edge node {} (3) (2) edge node {} (4) (3) edge node {} (4);
      \end{tikzpicture}
    }
  \caption{An acyclic digraph.}
  \label{fig:dag-implicitize}
\end{figure}
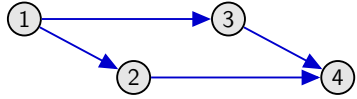

\begin{example}
  \label{ex:dag-implicitize}
  To illustrate the use of a different software, we change to {\sc
    Macaulay2} \cite{M2}.  The following code computes a generating
  set for the vanishing ideal of the graph $G$ shown in
  Figure~\ref{fig:dag-implicitize}:
  \begin{small}
\begin{verbatim}
R = QQ[l12,l13,l24,l34, s11,s12,s13,s14, s22,s23,s24, s33,s34, s44, 
       MonomialOrder => Eliminate 4];
Lambda = matrix{{1, -l12, -l13, 0},
                {0, 1, 0, -l24},
                {0, 0, 1, -l34},
                {0, 0, 0, 1}};
S = matrix{{s11, s12, s13, s14},
           {s12, s22, s23, s24},
           {s13, s23, s33, s34},
           {s14, s24, s34, s44}};
W = transpose(Lambda)*S*Lambda;
I = ideal{W_(0,1),W_(0,2),W_(0,3),W_(1,2),W_(1,3),W_(2,3)};
eliminate({l12,l13,l24,l34},I)
\end{verbatim}
  \end{small}
  The `GraphicalModels' package \cite{MR3073716} automates the computation:
  \begin{small}
\begin{verbatim}
needsPackage "GraphicalModels";
G = digraph {{1,{2,3}},{2,{4}},{3,{4}}};
R = gaussianRing G;
gaussianVanishingIdeal R
\end{verbatim}
  \end{small}
  Reproduced in our notation, the output shows that the ideal
  $\mathcal{I}(G)$ is generated by the two polynomials
  \begin{align}
    \label{eq:dag-generator1}
    f_1&=\sigma_{12}\sigma_{13}-\sigma_{11}\sigma_{23}, \\
        \label{eq:dag-generator2}
    f_2&=\sigma_{14}\sigma_{23}^2-\sigma_{13}\sigma_{23}\sigma_{24}-\sigma_{14}\sigma_{22}\sigma_{33}+\sigma_{12}\sigma_{24}\sigma_{33}+\sigma_{13}\sigma_{22}\sigma_{34}-\sigma_{12}\sigma_{23}\sigma_{34}.
  \end{align}
\end{example}

Computing $\mathcal{I}(G)$ using Gr\"obner bases is a method that
applies to any mixed graph but can be computationally prohibitive for
graphs with more than 5 or 6 nodes.  In order to solve model
equivalence problems combinatorial insight on
particular types of relations is needed.

\begin{example}
  Continuing with Example~\ref{ex:dag-implicitize}, observe that the
  two polynomials $f_1$ and $f_2$ from~(\ref{eq:dag-generator1})
  and~(\ref{eq:dag-generator2}) are determinants of submatrices of the
  covariance matrix $\Sigma$.  The following two displays
  locate the submatrices
  \begin{equation*}
    \begin{tikzpicture}
      \matrix [inner sep=1.5pt,matrix of math nodes,left delimiter={[},right
      delimiter={]},row sep=1ex,column sep=1ex,color=lightgray] (m) { {\color{black}\sigma_{11}} &
        \sigma_{12} & {\color{black}\sigma_{13}} & \sigma_{14}\\
        {\color{black}\sigma_{12}} & \sigma_{22} & {\color{black}\sigma_{23}} & \sigma_{24}\\
        \sigma_{13} & \sigma_{23} & \sigma_{33} & \sigma_{34}\\
        \sigma_{14} & \sigma_{24} & \sigma_{34} & \sigma_{44}\\
      }; 
      { \draw[color=black] (m-1-1.north west) -- (m-1-1.north east)
        -- (m-1-1.south east) -- (m-1-1.south west) -- (m-1-1.north
        west); }
    \end{tikzpicture}
    \qquad    \qquad
    \begin{tikzpicture}
      \matrix [inner sep=1.5pt,matrix of math nodes,left delimiter={[},right
      delimiter={]},row sep=1ex,column sep=1ex,color=lightgray] (m) {
        \sigma_{11} & {\color{black}\sigma_{12}} & {\color{black}\sigma_{13}} & {\color{black}\sigma_{14}}\\
        \sigma_{12} & {\color{black}\sigma_{22}} & {\color{black}\sigma_{23}} & {\color{black}\sigma_{24}}\\
        \sigma_{13} & {\color{black}\sigma_{23}} & {\color{black}\sigma_{33}} & {\color{black}\sigma_{34}}\\
        \sigma_{14} & \sigma_{24} & \sigma_{34} & \sigma_{44}\\
      }; { \draw[color=black] (m-2-2.north west) -- (m-2-3.north east)
        -- (m-3-3.south east) -- (m-3-2.south west) -- (m-2-2.north
        west); }
    \end{tikzpicture}
  \end{equation*}
  and show two boxes to visualize that each submatrix is obtained by
  bordering a principal submatrix by one additional row and column.
  The determinants $f_1$ and $f_2$ are seen to be almost principal
  minors of $\Sigma$.  As we discuss in Section~\ref{sec:cond-indep},
  the vanishing of almost principal minors of a Gaussian covariance
  matrix has a particular probabilistic meaning, namely, conditional
  independence
  \cite{MR3183760}.
\end{example}

\section{Conditional Independence}
\label{sec:cond-indep}

Suppose $X=(X_i:i\in V)$ is a Gaussian random vector indexed by a
finite set $V$ and with covariance matrix $\Sigma\in\mathit{PD}_V$.
Let $i,j\in V$ be two distinct indices, and let
$S\subseteq V\setminus\{i,j\}$.  The random variables $X_i$ and $X_j$
are conditionally independent given $X_S$ if and only if
$\det(\Sigma_{iS\times jS})=0$.  This connection between Gaussian
conditional independence and the vanishing of almost principal minors
of the covariance matrix is explained in detail in \cite[Chapter
3]{oberwolfach} and \cite{MR2362722}.

It is fully understood which conditional independence relations hold
in the covariance matrices of a linear structural equation model.  The
following concepts are needed to state the result.  Let $\pi$ be a
semi-walk from node $i$ to node $j$ in the considered mixed graph
$G=(V,D,B)$, and let node $k$ be a non-endpoint of $\pi$.  A
\emph{collider} on $\pi$ is a node $k$ that is an internal node on
$\pi$, and a head on the two edges that precede and succeed $k$ on
$\pi$.  We recall our convention that the two nodes incident to a
bidirected edge are both heads.  Pictorially, $\pi$ includes the
segment $\to k \leftarrow$, $\to k\bi$, $\bi k\leftarrow$ or
$\bi k\bi$.  A \emph{non-collider} on $\pi$ is an internal node of
$\pi$ that is not a collider on $\pi$.

\begin{definition}
  \label{def:dsep}
  Fix a set $S\subseteq V$.  Two nodes $i,j\in V$ are
  \emph{d-connected} by $S$ if $G$ contains a semi-walk from $i$ to
  $j$ that has all colliders in $S$ and all non-colliders outside $S$.
  If no such semi-walk exists, then $i$ and $j$ are \emph{d-separated}
  by $S$.
\end{definition}

The following theorem was first proven for acyclic digraphs
\cite{MR1064736} and later extended to cover arbitrary mixed graphs
\cite{smr:98}.

\begin{theorem}
  \label{thm:dsep}
  Let $G=(V,D,B)$ be a mixed graph.  Two nodes $i,j\in V$ are
  d-separated by $S$ if and only if
  $\det\left( \phi_G(\Lambda,\Omega)_{iS\times jS} \right)=0$
  for all $\Lambda\in\mathbb{R}^D_\mathrm{reg}$,
  $\Omega\in\mathit{PD}(B)$.
\end{theorem}

For acyclic digraphs, the theorem can be derived in a probabilistic
setup that extends to non-Gaussian conditional independence
\cite{lauritzen:1996}.  When lecturing about the result, the author
likes to discuss the example with three binary variables shown in
Figure~\ref{fig:boring-talk}.  In our linear Gaussian setting,
Theorem~\ref{thm:dsep} is a special case of
Theorem~\ref{thm:trek-separation} that we treat in more detail.

Define the \emph{conditional independence ideal}
\[
  \mathcal{I}_\text{CI}(G) \;=\; \langle \;\det\left({
      \Sigma}_{iS\times jS}\right) \::\: i,j \ \text{d-separated by }
  S\;\rangle.
\]
By Theorem~\ref{thm:dsep},
$\mathcal{I}_\text{CI}(G)\subseteq \mathcal{I}(G)$ for any mixed graph
$G$.  In Example~\ref{ex:dag-implicitize},
$\mathcal{I}_\text{CI}(G)= \mathcal{I}(G)$ but this may be false even
for acyclic digraphs \cite{MR2412156}.  Nevertheless, if $G$ is an
acylic digraph, then the set of covariance matrices $\mathcal{M}_G$ is
cut out by conditional independence relations.  In other words, for an
acyclic digraph,
\begin{equation}
  \label{eq:dag-model-cut-out}
  \mathcal{M}_G \;=\; V\left(\mathcal{I}_\text{CI}(G) \right) \cap \mathit{PD}_V
\end{equation}
is the positive definite part of the variety of the conditional
independence ideal.  The equality in~(\ref{eq:dag-model-cut-out}) also
holds when $G$ is an ancestral graph, as defined in
Section~\ref{sec:glob-ident}.  However, it is false in general as can
be seen in Example~\ref{ex:two-ivs}.  We remark that for acyclic
digraphs it has been proven that saturating the conditional
independence with respect to all principal minors yields the
vanishing ideal \cite{6875189}:
\[
  \mathcal{I}(G) \quad=\quad \mathcal{I}^\text{CI}(G)\::\:
  \bigg(\prod_{A\subset V} \det({
    \Sigma}_{A\times A})\bigg)^\infty.
\]
  

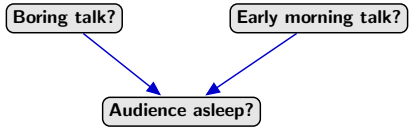
\begin{figure}[t]
  \begin{center}
    \scalebox{0.7}{
      \begin{tikzpicture}[->,>=triangle 45,shorten >=1pt, auto,thick,
        main
        node/.style={rectangle,fill=gray!20,draw,font=\sffamily\bfseries}]
        \node[main node,rounded corners] (2a) 
        {Audience asleep?}; \node[main node,rounded corners] (1a)
        [above left=1.2cm and -.4cm of 2a] {Boring talk?}; \node[main
        node,rounded corners] (3a) [above right=1.2cm and -.6cm of 2a]
        {Early morning talk?}; \path[color=black!20!blue,every
        node/.style={font=\sffamily\small}] (1a) edge node {} (2a)
        (3a) edge node {} (2a);
      \end{tikzpicture}
    }
  \end{center}
  \caption{A graph for three binary variables.}
  \label{fig:boring-talk}
\end{figure}

For acyclic digraphs and more generally ancestral graphs, the equality
from (\ref{eq:dag-model-cut-out}) allows us to answer
Question~\ref{qu:model-equivalence} on covariance equivalence by
checking whether two graphs have the same d-separation
relations.  The latter comparison can be accomplished in polynomial
time.  We state the result for acyclic digraphs
\cite{MR1096723,VermaPearl90}.  A generalization for ancestral graphs
was given more recently \cite{ali:2009}; see also \cite{zhao:2005}.

\begin{theorem}
  \label{thm:dag-equiv}
  Let $G$ and $G'$ be two acyclic digraphs.  Then
  $\mathcal{M}_G=\mathcal{M}_{G'}$ if and only if $G$ and $G'$ have
  the same adjacencies and the same unshielded colliders.  An
  unshielded collider is an induced subgraph of the form $i\to
  j\otblue k$. 
\end{theorem}
  
We remark that it can also be decided in polynomial time whether two
digraphs that are not necessarily acyclic have the same d-separation
relations \cite{rich:dcgmarkov}.  When the graphs have directed cycles
then d-separation equivalence is necessary but not sufficient for
covariance equivalence; see, e.g., the example in \cite{drton:lrt}.

\section{Trek Separation}
\label{sec:trek-separation}

We now turn to the characterization of the vanishing of general minors
of the covariance matrices in a linear structural equation model.  Our
first example clarifies the importance of minors that are not almost
principal.

\begin{example}
  \label{ex:tetrad}
  If $G$ is the graph from Figure~\ref{fig:two-ivs} and
  Example~\ref{ex:two-ivs}, then $\mathcal{I}(G)$ is 
  generated by $\det(\Sigma_{12,34})$.  Such off-diagonal $2\times 2$
  minors are known as tetrads in the statistical literature; see
  \cite{MR2299716} and the references therein.
\end{example}

The tetrad representation theorem gives a combinatorial
characterization of the vanishing of any $2\times 2$ determinant
\cite{spirtes:2000}.  The theorem has been greatly generalized, and we
now have a full combinatorial understanding of when a minor of the
covariance matrix vanishes based on the notion of trek-separation
\cite{sullivant:2010}.  Moreover, the non-vanishing determinants can
be described precisely \cite{MR3044565}.

\begin{definition}
  Let $A,C,S_A,S_C\subseteq V$ be sets of nodes of the mixed graph
  $G=(V,D,B)$.  The pair $(S_A,S_C)$ \emph{trek-separates} $A$ and $C$
  if every trek from $i\in A$ to $j\in C$ intersects $S_A$ on its left
  side or $S_C$ on its right side.
\end{definition}
  
\begin{theorem}
  \label{thm:trek-separation}
  Let $G$ be a mixed graph.  Then the $A\times C$ submatrix of
  $\phi_G(\Lambda,\Omega)$ has generic rank $\le r$ if and
  only if there exist sets $S_A$ and $S_C$ with $|S_A|+|S_C|\le r$
  such that $(S_A,S_C)$ trek-separates $A$ and $C$.
\end{theorem}

The theorem for acyclic mixed graphs is proven in
\cite{sullivant:2010}.  The cases with directed cycles are covered by
the results in \cite{MR3044565}.  We describe the ideas behind the proof.

\begin{proof}[Proof outline]
  The problem can be reduced to the case where $G$ is a digraph by
  subdividing bidirected edges.  For each edge $i\bi j\in B$ we
  introduce a new node $\{i,j\}$ and two edges $\{i,j\}\to i$ and
  $\{i,j\}\to j$.  The new digraph $G'$ thus has $|V|+|B|$ nodes and
  $|D|+2|B|$ edges.  If $G$ is the mixed graph from
  Figure~\ref{fig:iv:mixed}, then $G'$ is the digraph in
  Figure~\ref{fig:iv:dag}.  Every trek in $G$ corresponds to a trek in
  $G'$ where an edge $i\bi j$ in $G$ corresponds to
  $i\leftarrow \{i,j\}\to j$ in $G'$.  The entries of $(\phi_G)_{A,C}$
  and those of $(\phi_{G'})_{A,C}$ can be shown to admit the same set
  of relations. 

  In the sequel, assume that $G$ itself is a digraph.  Then
  $\mathit{PD}(B)$ contains diagonal matrices with positive entries.
  The rank of $\phi_G(\Lambda,\Omega)$ for
  $\Lambda\in\mathbb{R}^D_\mathrm{reg}$ and $\Omega\in\mathit{PD}(B)$
  is then the same as that of $\Sigma=\phi_G(\Lambda,I)$.

  To establish the claim, we may study the vanishing of the
  determinants of square submatrices.  So assume that $|A|=|C|=r+1$.
  The Cauchy-Binet formula gives that
  \begin{equation}
    \label{eq:cauchy-binet}
    \det\left( \Sigma_{A\times C}\right) \;=\; \sum_{E} \det\left( \left[
        (I-\Lambda)^{-1}\right]_{E\times A}\right) \det\left(
      \left[(I-\Lambda)^{-1}\right]_{E\times C}\right),
  \end{equation}
  where the sum is over subsets $E\subseteq V$ of cardinality
  $r+1$.  By the Lindstr\"om-Gessel-Viennot lemma for general
  digraphs, it holds that
  \begin{equation}
    \label{eq:LGV}
    \det\left( \left[
        (I-\Lambda)^{-1}\right]_{E\times
        A}\right) \;=\; 0 \quad \text{for all} \
    \Lambda\in\mathbb{R}^D_\mathrm{reg}, 
  \end{equation}
  if and only if no system of $r+1$ directed paths from $E$ to $A$ is
  vertex-disjoint.  Applying this to all terms
  in~(\ref{eq:cauchy-binet}) shows that
  $\det\left( \Sigma_{A\times C}\right)$ vanishes if and only if every
  system of treks from $A$ to $C$ has a sided intersection.  Here, an
  intersection on the left side of a trek corresponds to the vanishing
  of determinants of the matrices
  $ \left[ (I-\Lambda)^{-1}\right]_{E\times A}$ and, similarly,
  interesections on the right side are related to the determinants of
  the matrices $\left[(I-\Lambda)^{-1}\right]_{E\times C}$.  The
  characterization by sided intersections in trek systems can be
  turned into the claimed statement about trek-separation via Menger's
  theorem.  To account for the distinct role played by the left and
  the right sides of the treks, Menger's theorem is applied in a
  digraph $\tilde G$ that results from duplicating the nodes and edges of
  $G$.  Each node and each edge of $G$ has a left and a right side
  version in $\tilde G$.  Menger's theorem yields a cut set $S$ of
  cardinality $|S|\le r$ in $\tilde G$.  Partitioning $S$ according to
  the left and right side yields the pair $(S_A,S_C)$ for the claimed
  trek-separation.
\end{proof}

\begin{example}
  \label{ex:trek-sep1}
  When $G$ is the the graph from Figure~\ref{fig:two-ivs}, then the
  submatrix $[\phi_G(\Lambda,\Omega)]_{12,34}$ has generic rank 1;
  recall Example~\ref{ex:tetrad}.  Theorem~\ref{thm:trek-separation}
  shows that the rank is at least 1 because $(\emptyset,\emptyset)$
  does not trek-separate $\{1,2\}$ and $\{3,4\}$.  For instance, there
  is the trek $1\to 3$.  That the rank is no larger than 1 follows
  from $(\emptyset,\{3\})$ trek-separating $\{1,2\}$ and $\{3,4\}$.
  For instance, the node $3$ is on the right side of the two treks
  $1\to 3$ and $2\to 3\to 4$.
\end{example}

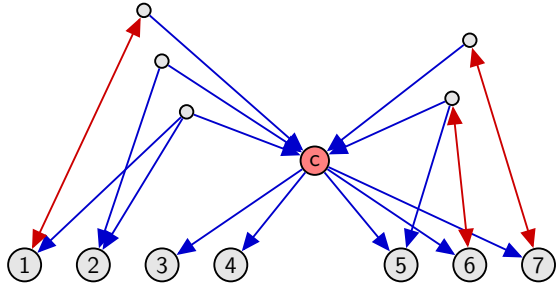
\begin{figure}[t]
  \centering
    \scalebox{0.9}{
      \begin{tikzpicture}[->,>=triangle 45, shorten >=0pt, auto,thick,
        main node/.style={circle,inner sep=2pt,
          fill=gray!20,draw,font=\sffamily}]

        \node[main node] (1) {1}; \node[main node] (2) [right=.5cm of
        1] {2}; \node[main node] (3) [right=.5cm of 2] {3}; \node[main
        node] (4) [right=.5cm of 3] {4}; \node[main node] (5)
        [right=2cm of 4] {5}; \node[main node] (6) [right=.5cm of 5]
        {6}; \node[main node] (7) [right=.5cm of 6] {7};

        \node[main node] (8) [above right=3.5cm and 1.5cm of 1] {};
        \node[main node] (9) [above right=2.75cm and .75cm of 2] {};
        \node[main node] (10) [above right=2cm and 0.1cm of 3] {};
        \node[main node] (11) [above left=2.2cm and 0.0cm of 6] {};
        \node[main node] (12) [above=2.95cm of 6] {};

        \node[main node,fill=red!50] (c) [above right=1.2cm and .9cm
        of 4] {c};

        \path[color=black!20!blue,every
        node/.style={font=\sffamily\small}] (8) edge (c) (9) edge (c)
        (9) edge (2) (10) edge (c) (10) edge (2) (10) edge (1) (c)
        edge (3) (c) edge (4) (c) edge (5) (c) edge (6) (c) edge (7)
        (11) edge (5) (11) edge (c) (12) edge (c);
        \path[color=black!20!red,<->,every
        node/.style={font=\sffamily\small}] (8) edge node {} (1) (12)
        edge node {} (7) (11) edge node {} (6);
      \end{tikzpicture}
    } \smallskip
  \caption{A `spider graph' with $\{1,2,3,4\}\times \{5,6,7\}$
    submatrix of rank two.}
  \label{fig:spider}
\end{figure}

\begin{example}
  \label{ex:trek-sep2}
  What is the generic rank of the $A\times C$ of
  $\phi_G(\Lambda,\Omega)$ when $G$ is the `spider graph' from
  Figure~\ref{fig:spider}, $A=\{1,2,3,4\}$ and $C=\{5,6,7\}$?  The
  node $c$ is on every trek between $A$ and $C$.  It follows that
  $(\{c\},\{c\})$ trek-separates $A$ and $C$ and thus the rank is no
  larger than two.  In fact, the rank is equal to two.  Consider the
  two treks
  \begin{align*}
    &\pi_1:1\bi \circ \toblue c \toblue 5,\\
    &\pi_2:3\otblue c \otblue \circ \bi 6.
  \end{align*}
  They have only node $c$ in common but $c\in\rhs{\pi_1}$ and
  $c\in\lhs{\pi_2}$.  Hence, a pair of trek-separating sets must use at
  least two nodes.
\end{example}

Trek separation solves the problem of characterizing the vanishing of
determinants.  However, there is currently no efficient method for
deciding when two mixed graphs are trek separation equivalent.


\section{Verma Constraints}
\label{sec:verma-constraints}

Much interesting ground lies beyond determinants of the covariance
matrix.  We content ourselves with two examples concerning a relation
first presented in
\cite{VermaPearl90}.  

\begin{example}
  \label{ex:verma-decompose}
  Let $G$ be the graph from Figure~\ref{fig:recap}.  The graph has no
  trek-separation relations as can be checked with the package
  `GraphicalModels' for {\sc Macaulay2} \cite{MR3073716}.
  The commands
  \begin{small}
\begin{verbatim}
needsPackage "GraphicalModels";
G = mixedGraph(digraph {{1,{2,3}},{2,{3}},{3,{4}}}, bigraph {{2,4}});
R = gaussianRing G;
trekIdeal(R,G)
\end{verbatim}
  \end{small}
  return the zero ideal.  Issuing the command
  \begin{small}
\begin{verbatim}
gaussianVanishingIdeal R 
\end{verbatim}
  \end{small}
  shows that $\mathcal{I}(G)$ is generated by
  \begin{multline}
    \label{eq:f-verma}
    f_\mathrm{Verma}=
    \sigma_{11} \sigma_{13} \sigma_{22} \sigma_{34}-\sigma_{11}
    \sigma_{13} \sigma_{23} \sigma_{24}-\sigma_{11} \sigma_{14}
    \sigma_{22} \sigma_{33}+\sigma_{11} \sigma_{14}
    \sigma_{23}^2\\
    -\sigma_{12}^2 \sigma_{13} \sigma_{34}
    +\sigma_{12}^2
    \sigma_{14} \sigma_{33}+\sigma_{12} \sigma_{13}^2
    \sigma_{24}-\sigma_{12} \sigma_{13} \sigma_{14} \sigma_{23}. 
  \end{multline}
  Clearly, $f_\mathrm{Verma}$ is not a
  determinant.  
  Its vanishing can be explained as follows.  Decompose $G$ into its
  mixed components.  The largest component $G[\{2,4\}]$ is depicted in
  Figure~\ref{fig:verma-comp}.  Appealing to
  Theorem~\ref{thm:tian-decomposition}, there is a rational function
  $\tau_{\{2,4\}}:\mathcal{M}_G\to \mathcal{M}_{G'}$, so
  $\tau_{\{2,4\}}$ is the covariance matrix for $G[\{2,4\}]$.  In
  $G[\{2,4\}]$, there is no trek from node 1 to node 4.  Hence, the
  $(1,4)$ entry of $\tau_{\{2,4\}}(\Sigma)$ is zero.  Clearing the
  denominator yields $f_\mathrm{Verma}(\Sigma)=0$ for
  $\Sigma\in\mathcal{M}_G$. 
\end{example}

\begin{figure}[t]
  \centering
            \scalebox{0.9}{
      \begin{tikzpicture}[->,>=triangle 45, shorten >=0pt, auto,thick,
        main node/.style={circle,inner
          sep=2pt,fill=gray!20,draw,font=\sffamily}]
        
        \node[main node] (1) {1}; \node[main node] (2) [below
        right=1cm and 1cm of 1] {2}; \node[main node] (3) [right=
        2cm of 1] {3}; \node[main node] (4) [below 
        right=1cm and 1cm of 3] {4};
        
        \path[color=black!20!blue,every
        node/.style={font=\sffamily\small}] (1) edge (2)  (3) edge (4);
        \path[color=black!20!red,<->,every
        node/.style={font=\sffamily\small}] (2) edge (4);
      \end{tikzpicture}
      

      }
  \caption{The largest mixed component of the Verma graph.}
  \label{fig:verma-comp}
\end{figure}
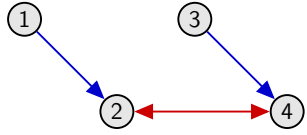

The example suggests that new relations can be discovered by
decomposing the graph and studying d-separation or trek-separation
relations in the mixed components.  The matter is more delicate, however.

\begin{example}
  \label{ex:verma-sink-decompose}
  In order to prevent decomposition of the Verma graph, add a fifth
  node and bidirected edges such that $(V,B)$ becomes connected.
  Specifically, consider the graph $G$ from
  Figure~\ref{fig:verma-sink-decomp}.  The new graph $G$ is simple and
  acyclic and, thus, has expected dimension $13$.  Because the nodes 2 and 5
  are d-separated by 1, we have
  $\sigma_{12}\sigma_{15}-\sigma_{11}\sigma_{25}\in\mathcal{I}(G)$.
  Other relations must exist, and indeed a Gr\"obner basis computation
  shows that
  \[
  \mathcal{I}(G) \;=\; \langle
  \sigma_{12}\sigma_{15}-\sigma_{11}\sigma_{25},\; f_\mathrm{Verma}
  \rangle : \sigma_{11}^\infty
  \]
  with $f_\mathrm{Verma}$ being the polynomial
  from~(\ref{eq:f-verma}).  The fact that
  $f_\mathrm{Verma}\in\mathcal{I}(G)$ is explained by
  Theorem~\ref{thm:ancestral-subgraphs}.  Since no directed edge of
  $G$ has a tail at node 5, the theorem allows us to consider the
  subgraph induced by the set of nodes $\{1,2,3,4\}$.  We are back in
  Example~\ref{ex:verma-decompose} and decomposition yields
  $f_\mathrm{Verma}$ as a relation.
\end{example}

It is clear that more contrived examples can be constructed in which
d-/trek-separation applies only after several rounds of alternating between
graph decomposition and forming a subgraph induced by an ancestral
set.  An overview of what is known about such a recursive approach can
be found in \cite{EvansEtAl14}, where the focus is on non-linear
models and manipulation of probability density functions.

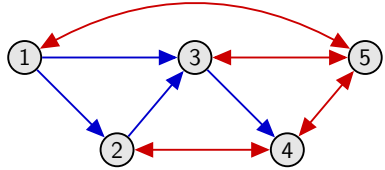
\begin{figure}[t]
  \centering 
  \scalebox{0.9}{
    \begin{tikzpicture}[->,>=triangle 45, shorten >=0pt, auto,thick,
      main node/.style={circle,inner
        sep=2pt,fill=gray!20,draw,font=\sffamily}]
      
      \node[main node] (1) {1}; \node[main node] (2) [below
      right=1cm and 1cm of 1] {2}; \node[main node] (3) [right=
      2cm of 1] {3}; \node[main node] (4) [below right=1cm and
      1cm of 3] {4}; \node[main node] (5) [right=2cm of 3] {5}; 
      
      \path[color=black!20!blue,every
      node/.style={font=\sffamily\small}] (1) edge  (2) (1) edge
      (3) (2) edge
      (3) (3) edge
      (4);
      \path[color=black!20!red,<->,every
      node/.style={font=\sffamily\small}] (2) edge (4)
      (5) edge (3) (5) edge (4);
      \path[color=black!20!red,<->,bend right,every
      node/.style={font=\sffamily\small}]  (5) edge (1);
    \end{tikzpicture}
  }
  \caption{A graph that can be decomposed only after
    removing the sink node 5.}
  \label{fig:verma-sink-decomp}
\end{figure}

\addtocontents{toc}{\vspace{.3cm}}{}{}
\section{Conclusion}
\label{sec:discussion}

Linear structural equation models have covariance matrices with rich
algebraic structure.  As we showed in
Section~\ref{sec:questions-interest}, statistical considerations
motivate a host of different algebraic problems.  In this review, we
focused on methods for parameter identification as well as relations
among covariances.  While much progress has been made, and continues
to be made \cite{fink:2016}, we have encountered a plethora of open
problems of algebraic and combinatorial nature.

In our review, we focused exclusively on linear and Gaussian models.
As noted repeatedly, many of the questions have interesting
generalizations to non-linear or non-Gaussian models.  In particular,
in settings with discrete random variables, as considered in
\cite{evans:2015:allconstraints,evans:2015}, algebra and Gr\"obner
basis techniques continue to be useful \cite{oberwolfach}.  Similarly,
many additional challenges arise in models with explicit latent
variables, which motivate studying the map $\phi_G$ when projected
onto a principal submatrix (recall
Section~\ref{sec:induced-subgraphs}).


\bibliographystyle{amsplain}
\bibliography{SEM_lectures}

\end{document}